\documentclass[onefignum,onetabnum]{siamart190516}

\usepackage{lipsum}
\usepackage{amsfonts}
\usepackage{graphicx}
\usepackage{epstopdf}
\usepackage{algorithmic}
\usepackage{amssymb}
\usepackage{makecell}
\ifpdf
  \DeclareGraphicsExtensions{.eps,.pdf,.png,.jpg}
\else
  \DeclareGraphicsExtensions{.eps}
\fi

\newsiamremark{remark}{Remark}
\newsiamremark{hypothesis}{Hypothesis}
\crefname{hypothesis}{Hypothesis}{Hypotheses}
\newsiamthm{claim}{Claim}

\newtheorem{ass}{Assumption}
\newtheorem{defn}{Definition}
\headers{Higher order approximation for stochastic wave equation}{Xing Liu and Weihua Deng}

\title{Higher order approximation for stochastic space fractional wave equation forced by an additive space-time Gaussian noise\thanks{Submitted to the editors DATE.
\funding{This work was supported by the National Natural Science Foundation of China under Grant
No. 11671182, and the AI and Big Data Funds under Grant No. 2019620005000775.}}}

\author{Xing Liu\thanks{School of Mathematics and Statistics, Gansu Key Laboratory of Applied Mathematics and Complex
Systems, Lanzhou University, Lanzhou 730000, People's Republic of China
  (\email{2718826413@qq.com}).}
\and Weihua Deng\thanks{Corresponding author. School of Mathematics and Statistics, Gansu Key Laboratory of Applied Mathematics and Complex
Systems, Lanzhou University, Lanzhou 730000, People's Republic of China
  (\email{dengwh@lzu.edu.cn}).}}

\usepackage{amsopn}

\ifpdf
\hypersetup{
  pdftitle={An Example Article},
  pdfauthor={D. Doe, P. T. Frank, and J. E. Smith}
}
\fi

\begin{document}

\maketitle

\begin{abstract}
The infinitesimal generator (fractional Laplacian) of a process obtained by subordinating a killed Brownian motion catches the power-law attenuation of wave propagation.  This paper studies the numerical schemes for the stochastic wave equation with fractional Laplacian as the space operator, the noise term of which is an infinite dimensional Brownian motion or fractional Brownian motion (fBm).
Firstly, we establish the regularity of the mild solution of the stochastic fractional wave equation. Then a spectral Galerkin method is used for the approximation in space, and the space convergence rate is improved by postprocessing the infinite dimensional Gaussian noise.
In the temporal direction, when the time derivative of the mild solution is bounded in the sense of mean-squared $L^p$-norm, we propose a modified stochastic trigonometric method, getting a higher strong convergence rate than the existing results, i.e., the time convergence rate is bigger than $1$. Particularly, for time discretization, the provided method can achieve an order of $2$ at the expenses of requiring some extra regularity to the mild solution. The theoretical error estimates are confirmed by numerical experiments.
\end{abstract}



\section{Introduction}

The wave propagation in ideal medium is well described by the classical wave equation $\partial^2 u(x,t)/ \partial t^2=\mathrm{\Delta}u(x,t)$. However, sometimes the classical wave equation fails to model the wave propagations in complex inhomogeneous media (e.g., viscous damping in the seismic isolation of buildings, medical ultrasound, and seismic wave propagation \cite{1,4,5,8}), because of their power-law attenuations. One of the most effective ways to characterize the wave propagation with power-law attenuations is to resort to the nonlocal operator --- the infinitesimal generator (fractional Laplacian) of a process obtained by subordinating a killed Brownian motion.

Currently, two stochastic processes are very popular:  one is killed subordinate Brownian motion, and the other is subordinate killed Brownian motion.
Let $D$ be a bounded region, $B(t)$ be a Brownian motion with $B(0)\in D$, and $\tau_D=\inf\{t>0: B(t)\notin D\}$. Denote $T_t$ as an $\alpha$-stable subordinator. The first stochastic process (killed subordinate Brownian motion) \cite{3} is defined as
\begin{equation*}
 X_1(t)=\left\{
\begin{array}{cc}
B(T_t),\quad & t<\tau_D,\\
\Theta,\quad & t\ge\tau_D,
\end{array}
 \right.
 \end{equation*}
where $\Theta$ is a coffin state, meaning that the subordinate Brownian motion will be killed when first leaving the domain $D$; while the second stochastic process (subordinate killed Brownian motion) \cite{7} is
\begin{equation*}
 X_2(t)=\left\{
\begin{array}{cc}
B(T_t),\quad &T_t<\tau_D,\\
\Theta,\quad &T_t\ge\tau_D
\end{array}
 \right.
 \end{equation*}
with $\Theta$ still being a coffin state, implying to subordinate a killed Brownian motion (when first leaving the domain $D$).
The infinitesimal generator of $X_1(t)$ has the form
\begin{equation*}
    (-\Delta_1)^{\alpha}u(x)=c_{n,\alpha}{\rm P.V.}\int_{\mathbb{R}^n}\frac{u(x)-u(y)}{|x-y|^{n+2\alpha}}dy,\quad \alpha\in (0,1),
\end{equation*}
where  $c_{n,\alpha}=\frac{2^{2\alpha}\alpha\Gamma(n/2+\alpha)}{\pi^{n/2}\Gamma(1-\alpha)}$, ${\rm P.V.}$ means the principal value integral, and $u(y)=0$ for $y \in \mathbb{R}^n \backslash D$. Denote the infinitesimal generator of $X_2(t)$ as $(-\Delta)^{\alpha}$ and $-\Delta$ the infinitesimal generator of killed Brownian motion. It shows that  \cite{6,7}  if $\{(\lambda_i,\phi_i)\}^\infty_{i=1}$ are the eigenpairs of $-\mathrm{\Delta}$, then $\{(\lambda^\alpha_i,\phi_i)\}^\infty_{i=1}$ are the eigenpairs of $(-\mathrm{\Delta})^\alpha$, i.e.,
 \begin{equation}\label{eq:1.03}
 \left\{
\begin{array}{cc}
 -\mathrm{\Delta}\phi_i=\lambda_i\phi_i, \quad &\mathrm{in} \ D,\\
  \quad \phi_i=0,\quad &\mathrm{on} \ \partial D,
 \end{array}
 \right.
 \end{equation}
and
\begin{equation}\label{eq:1.04}
 \left\{
\begin{array}{cc}
(-\mathrm{\Delta})^\alpha\phi_i =\lambda^\alpha_i\phi_i, \quad &\mathrm{in} \ D,\\
  \quad \phi_i=0,\quad &\mathrm{on} \  \partial D.
 \end{array}
 \right.
 \end{equation}

The operator used in this paper is the one defined in  \eqref{eq:1.04}.
Moreover, we are also concerned with the external noises that possibly affect the wave propagation. Two most popular external noises are white noise and fractional Gaussian noise, both of which are considered in this paper. 
The fractional Gaussian noise is defined as the formal derivative of the fractional Brownian motion (fBm) $\beta_H(t)$, which is Gaussian process with an index $H\in(0,1)$. The fBm  has two unique properties: self-similarity and stationary increments \cite{26,27}. As $H=\frac{1}{2}$, the fBm reduces to a standard Brownian motion. The formal derivative of Brownian motion is white noise. For $H\ne\frac{1}{2}$, unlike
Brownian motion, the fBm exhibits long-range dependence: the behavior of the process after a given time $t$ depends on the situation at $t$ and the whole history
of the process up to time $t$ \cite{22}. According to the properties of the fBm and Brownian motion, one can choose the appropriate noise in practical applications.

With the above introduction of nonlocal operator and the external noise, the model we discuss in this paper is the stochastic wave equation
\begin{equation}\label{eq:1.05}
\left\{
\begin{array}{ll}
\frac{\mathrm{d} \dot{u}(x,t)}{\mathrm{d} t}
=-(-\mathrm{\Delta})^\alpha u(x,t)+f\left(u(x,t)\right)+\dot{B}_H(x,t) \quad \mathrm{in} \ D\times(0,T],\\[1.5mm]
u(x,0)=u_0, \  \dot{u}(x,0)=v_0  
\quad  \mathrm{in}\ D,\\[1.5mm]
u(x,t)=0, \quad \mathrm{in}\ \partial D\times(0,T],
\end{array}
\right.
\end{equation}
where $\dot{u}(x,t)$ is the first order time derivative of $u(x,t)$, $d/dt$ means the partial derivative with respect to $t$, $f$ is the source term, $D\subset \mathbb{R}^d\ (d=1,2,3)$, and $\dot{B}_H(x,t)$ is the formal derivative of the infinite dimensional space-time Gaussian process $B_H(x,t)$ with $0<\alpha\le1$ and $\frac{1}{2}\le H<1$.

Over the last few decades, there is much progress in both strong and weak approximations of the stochastic wave equation driven by the space-time white noise. A full discretization of the stochastic wave equation driven by additive space-time white noise is presented with a spectral Galerkin approximation in space and a temporal approximation by exponential time integrators involving linear functionals of the white noise \cite{10}. In \cite{11,9}, the stochastic trigonometric method for solving the stochastic wave equation with multiplicative space-time white noise is studied in time. The work \cite{12} investigates a discrete approximation for the stochastic space-time fractional wave equation forced by an additive space-time white noise. In this paper, besides the additive white noise, the additive fractional Gaussian noise is also discussed. Moreover, if $\dot{u}(x,t)$  
is bounded in the sense of mean-squared $L^p$-norm, then by modifying the stochastic trigonometric method for time discretization, we can obtain a high order convergence rate. 
In particular, as $H=\frac{1}{2}$, we use the independent increment property of Brownian motion to obtain the optimal temporal error estimate; for $H\in\left(\frac{1}{2},1\right)$, by using the covariance of stochastic integral for fBm (Lemma \ref{le:02}), we obtain the optimal error estimate in time. For space approximation, the spectral Galerkin scheme is used; and the space convergence rate is improved by postprocessing the additive Gaussian noise.

This paper is organized as follows. In the next section, we introduce some notations and preliminaries, including assumptions and properties of fBm. In Section \ref{sec:3}, by using the Dirichlet eigenpairs, we present the regularity of the mild solution $u(x,t)$ and the time derivative $\dot{u}(x,t)$ 
in the sense of mean-squared $L^p$-norm. In Section \ref{sec:4}, the spectral Galerkin spatial semidiscretization of \eqref{eq:1.05} and the postprocessing approach of the additive space-time Gaussian noise are discussed. In Section \ref{sec:5}, we modify the stochastic trigonometric method to obtain a high order temporal discretization of \eqref{eq:1.05}; and the convergence order for the proposed fully discrete scheme is derived. The numerical experiments are performed in Section \ref{sec:6}. We end the paper with some discussions in Section \ref{sec:7}.

\section{Notations and preliminaries}
\label{sec:2}

In this section, we gather preliminary results on the Dirichlet eigenpairs and fBm, which are commonly used in the paper.

Let $U=L^2(D;\mathbb{R})$ be a real separable Hilbert space with $L^2$ inner product $\langle\cdot,\cdot\rangle$ and the corresponding induced norm $\|\cdot\|$. We define the unbounded linear operator $A^\nu$ by $A^\nu u=\left(-\mathrm{\Delta}\right)^\nu u$ on the domain
\begin{equation*}
\mathrm{dom}\left(A^{\nu}\right)=\left\{ A^\nu u\in U:u(x)=0,\ x\in \partial D \right\}.
\end{equation*}
Then Equation \eqref{eq:1.04} implies that
\begin{equation*}
A^{\frac{\nu}{2}}\phi_{i}(x)=\lambda^{\frac{\nu}{2}}_{i}\phi_{i}(x)
\end{equation*}
and
\begin{equation*}
A^{\frac{\nu}{2}}u=\sum^\infty_{i=1}\lambda^{\frac{\nu}{2}}_{i}\left\langle u,\phi_{i}(x)\right\rangle\phi_{i}(x),
\end{equation*}
where $\phi_{i}(x)$, $i=1,2,\dots,$ denote the normalized eigenfunctions of the fractional Laplacian operator $\left(-\mathrm{\Delta}\right)^{\frac{\nu}{2}}$, and $\lambda^{\frac{\nu}{2}}_{i}$, $i=1,2,\dots,$ are the corresponding eigenvalues.
Moreover, we define the Hilbert space $\dot{U}^\nu=\mathrm{dom}\left(A^{\frac{\nu}{2}}\right)$
equipped with the inner product
\begin{equation*}
\left\langle u, v\right\rangle_{\nu}=\sum^\infty_{i=1}\lambda^{\frac{\nu}{2}}_{i}\left\langle u,\phi_{i}(x)\right\rangle\times\lambda^{\frac{\nu}{2}}_{i}\left\langle v,\phi_{i}(x)\right\rangle
\end{equation*}
and norm
\begin{equation*}
\|u\|_\nu=\left(\sum^\infty_{i=1}\lambda^{\nu}_{i}\left\langle u,\phi_{i}(x)\right\rangle^2\right)^{\frac{1}{2}}.
\end{equation*}
In particular, $\dot{U}^0=U$.

\begin{lemma} [\cite{13,14,15}] \label{le:01}
Let $\Omega$ denote a bounded domain in $\mathbb{R}^d$, $d\in\{1,2,3\}$, and $|\Omega|$ the volume of $\Omega$. Let $\lambda_{i}$ be the i-th eigenvalue of the Dirichlet homogeneous boundary problem for the Laplacian operator $-\mathrm{\Delta}$ in $\Omega$. 
Then
\begin{equation*}
C_0i^{\frac{2}{d}}\le\lambda_{i}\le C_1i^{\frac{2}{d}},
\end{equation*}
where $i\in\mathbb{N}$, and the constants $C_0$ and $C_1$ are independent of $i$.
\end{lemma}

\begin{ass}\label{as:2.1}
The function $f:U\to U$ in (\ref{eq:1.05}) satisfies
\begin{equation*}
\|f(u)-f(v)\| \lesssim \|u-v\| ~ {\rm for ~any} ~u,v\in U,
\end{equation*}
and
\begin{equation*}
\|A^{\frac{\nu}{2}} f(u)\|\lesssim 1+\|A^{\frac{\nu}{2}} u\|  ~ {\rm for } ~ u\in \dot{U}^\nu~ {\rm with } ~ \nu\ge0.
\end{equation*}
\end{ass}
For later use, we collect concepts of fBm; for more details, one
can refer to \cite{16,23,21,18,17}.
\begin{defn}\label{de:2.1}
Let $\beta_H(t)$ be the two-sided one-dimensional fBm with Hurst index $H\in(0,1)$ and $t\in\mathbb{R}$. The stochastic process $\beta_H(t)$ is characterized by the properties:

$\mathrm{(i)}$
$\beta_H(0)=0$;

$\mathrm{(ii)}$
$\mathrm{E}\left[\beta_H(t)\right]=0$, \ $t\in\mathbb{R}$;

$\mathrm{(iii)}$
$\mathrm{E}\left[\beta_H(t)\beta_H(s)\right]=\frac{1}{2}\left(|t|^{2H}+|s|^{2H}-|t-s|^{2H}\right)$, \ $t,\ s\in\mathbb{R}$,
\end{defn}
where $\mathrm{E}$ denotes the expectation. As $H=\frac{1}{2}$, $\beta_H(t)$ is a standard Brownian motion, being a process with independent increment.

\begin{ass}\label{as:2.2}
Let driven stochastic process $B_{H}(x,t)$ be a cylindrical fBm with respect to the normal filtration $\{\mathcal{F}_t\}_{t\in[0,T]}$. The
infinite dimensional space-time stochastic process can be represented by the formal series
\begin{equation*}
B_{H}(x,t)=\sum^\infty_{i=1}\sigma_{i}\beta^{i}_{H}(t)\phi_{i}(x),
\end{equation*}
where $|\sigma_{i}|\lesssim\lambda_i^{-\rho}$ ($\rho\ge0$, $\lambda_i$ is given in Lemma \ref{le:01}), $\beta^{i}_{H}(t)$, $i=1,2,\dots,$ are mutually independent real-valued fractional Brownian motions with $\frac{1}{2}\le H<1$, and $\{\phi_{i}(x)\}_{i\in\mathbb{N}}$ is an orthonormal basis of $U$.
\end{ass}

We define $L^p(D,\dot{U}^\nu)$ to be the separable Hilbert space of $p$-times integrable random variables with norm
\begin{equation*}
\left\|u\right\|_{L^p(D,\dot{U}^\nu)}=\left(\mathrm{E}\left[\left\|u\right\|^p_{\nu}\right]\right)^{\frac{1}{p}}, \quad \nu\ge0.
\end{equation*}

\begin{lemma} [\cite{17}]\label{le:02}
For $f, g\in L^2(\mathbb{R};\mathbb{R})\cap L^1(\mathbb{R};\mathbb{R})$, as $\frac{1}{2}<H<1$, we have
\begin{equation*}
\mathrm{E}\left[\int_Rf(s)\mathrm{d}\beta_H(s)\right]=0
\end{equation*}
and
\begin{equation*}
\mathrm{E}\left[\int_Rf(s)\mathrm{d}\beta_H(s)\int_Rg(t)\mathrm{d}\beta_H(t)\right]=H(2H-1)\int_R\int_R\mathrm{E}\left[f(s)g(t)\right]|s-t|^{2H-2}\mathrm{d}s\mathrm{d}t.
\end{equation*}
\end{lemma}

For $H\in\left(\frac{1}{2},1\right)$ and that $\int^t_sf(r)(r-s)^{H-\frac{3}{2}}\left(\frac{s}{r}\right)^{\frac{1}{2}-H}\mathrm{d}r$ belongs to $L^2\left([0,T],\mathbb{R}\right)$, an expression of the Wiener integral with respect to fBm is introduced as \cite{19,20}
\begin{equation*}
\int^t_0f(s)\mathrm{d}\beta_H(s)=C_H\int^t_0\int^t_s\left(H-\frac{1}{2}\right)f(r)(r-s)^{H-\frac{3}{2}}\left(\frac{s}{r}\right)^{\frac{1}{2}-H}\mathrm{d}r\mathrm{d}\beta(s),\ s,t\in[0,T],
\end{equation*}
where $\beta(t)$ is standard Brownian motion, and
\begin{equation*}
C_H=\left(\frac{2H\times\Gamma(\frac{3}{2}-H)}{\Gamma(H+\frac{1}{2})\Gamma(2-2H)}\right)^{\frac{1}{2}}.
\end{equation*}

\section{Regularity of the solution} \label{sec:3}
To begin with we can give a system of equations by coupling \eqref{eq:1.05} and $\mathrm{d}u(x,t)=\dot{u}(x,t)\mathrm{d}t$. The system of equations is beneficial to analyze the regularity of the mild solution of \eqref{eq:1.05}, including existence, uniqueness, and time H\"older continuity. Moreover, the system of equations is transformed into an equivalent form, which will be used to obtain the approximation of \eqref{eq:1.05}.

In the interest of brevity and readability, we use the following equation instead of \eqref{eq:1.05}
\begin{equation}\label{eq:3.1}
\left\{
\begin{array}{ll}
\mathrm{d} \dot{u}(t)=-A^\alpha  u(t)\mathrm{d}t+f\left(u(t)\right)\mathrm{d}t+\mathrm{d}B_H(t),& \quad \mathrm{in} \ D\times(0,T],\\[1mm]
u(0)=u_0, \dot{u}(0)=v_0& \quad  \mathrm{in}\ D,\\[1mm]
u(t)=0,& \quad \mathrm{in}\ \partial D,
\end{array}
\right.
\end{equation}
where $u(t)=u(x,t)$ and $B_H(t)=B_H(x,t)$. Let $v(t)=\dot{u}(t)$. Then

\begin{equation}\label{eq:3.2}
\mathrm{d}X(t)=\Lambda X(t)\mathrm{d}t+
\left[\begin{array}{c}
0\\
f(u(t))
\end{array} \right]\mathrm{d}t+
\left[\begin{array}{c}
0\\
I
\end{array} \right]\mathrm{d}B_H(t),
\end{equation}
where
\begin{equation*}
X(t)=
\left[\begin{array}{c}
u(t)\\
v(t)
\end{array} \right],\qquad
\Lambda=
\left[\begin{array}{cc}
0& I\\
-A^{\alpha}&0
\end{array}\right].
\end{equation*}
Then a formal mild solution $X(t)$ for \eqref{eq:3.2} is given as
\begin{equation}\label{eq:3.3}
X(t)=\mathrm{e}^{\Lambda t}X(0)+\int^t_0\mathrm{e}^{\Lambda (t-s)}
\left[\begin{array}{c}
0\\
f\left(u(s)\right)
\end{array} \right]\mathrm{d}s+\int^t_0\mathrm{e}^{\Lambda (t-s)}\left[\begin{array}{c}
0\\
I
\end{array} \right]\mathrm{d}B_H(s),
\end{equation}
where $\mathrm{e}^{\Lambda t}$ can be expressed as
\begin{equation}\label{eq:3.3-1}
\mathrm{e}^{\Lambda t}=
\left[\begin{array}{cc}
\cos\left(A^{\frac{\alpha}{2}}t\right)&A^{-\frac{\alpha}{2}}\sin\left(A^{\frac{\alpha}{2}}t\right)\\[1mm]
-A^{\frac{\alpha}{2}}\sin\left(A^{\frac{\alpha}{2}}t\right)&\cos\left(A^{\frac{\alpha}{2}}t\right)
\end{array}\right].
\end{equation}
The definitions of cosine operator $\cos\left(A^{\frac{\alpha}{2}}t\right)$ and sine operator $\sin\left(A^{\frac{\alpha}{2}}t\right)$ are given in Appendix \ref{sec:A}. Substituting \eqref{eq:3.3-1} into \eqref{eq:3.3}, then two components of $X(t)$ are obtained as
\begin{equation}\label{eq:3.4}
\begin{split}
u(t)=&\cos\left(A^{\frac{\alpha}{2}}t\right)u_0+A^{-\frac{\alpha}{2}}\sin\left(A^{\frac{\alpha}{2}}t\right)v_0+\int^t_0A^{-\frac{\alpha}{2}}\sin\left(A^{\frac{\alpha}{2}}(t-s)\right)f\left(u(s)\right)\mathrm{d}s\\[1.5mm]
&+\int^t_0A^{-\frac{\alpha}{2}}\sin\left(A^{\frac{\alpha}{2}}(t-s)\right)\mathrm{d}B_H(s),\\[1.5mm]
v(t)=&-A^{\frac{\alpha}{2}}\sin\left(A^{\frac{\alpha}{2}}t\right)u_0+\cos\left(A^{\frac{\alpha}{2}}t\right)v_0+\int^t_0\cos\left(A^{\frac{\alpha}{2}}(t-s)\right)f\left(u(s)\right)\mathrm{d}s\\[1.5mm]
&+\int^t_0\cos\left(A^{\frac{\alpha}{2}}(t-s)\right)\mathrm{d}B_H(s).
\end{split}
\end{equation}

In order to obtain the regularity of $u(t)$ and $v(t)$, we need to consider the regularity estimate of the stochastic integral in \eqref{eq:3.4}. For $\frac{1}{2}\le H<1$ and $p>1$, the Burkh\"older-Davis-Gundy inequality \cite{25,24} implies
\begin{eqnarray}\label{eq:3.5}
&\mathrm{E}&\left[\left\|\int^t_0A^{\frac{\gamma-\alpha}{2}}\sin\left(A^{\frac{\alpha}{2}}(t-s)\right)\mathrm{d}B(s)\right\|^p\right]\\
&\le&C_p\left(\int^t_0\sum^\infty_{i=1}\left|\lambda_i^{\frac{\gamma-\alpha-2\rho}{2}}\sin\left(\lambda_i^{\frac{\alpha}{2}}(t-s)\right)\right|^2\mathrm{d}s\right)^{\frac{p}{2}}\nonumber\\
&\le&C_p\left(t\sum^\infty_{i=1}i^{\frac{2(\gamma-\alpha-2\rho)}{d}}\right)^{\frac{p}{2}}\nonumber
 \end{eqnarray}
 and
 \begin{equation}\label{eq:3.5-1}
 \begin{split}
&\mathrm{E}\left[\left\|\int^t_0A^{\frac{\gamma-\alpha}{2}}\sin\left(A^{\frac{\alpha}{2}}(t-s)\right)\mathrm{d}B_H(s)\right\|^p\right]\\
&=\mathrm{E}\left[\left\|C_H\left(H-\frac{1}{2}\right)\int^t_0\int^t_sA^{\frac{\gamma-\alpha}{2}}\sin\left(A^{\frac{\alpha}{2}}(t-r)\right)\left(\frac{s}{r}\right)^{\frac{1}{2}-H}(r-s)^{H-\frac{3}{2}}\mathrm{d}r\mathrm{d}B(s)\right\|^p\right]\\
&\le C^p_H\left(H-\frac{1}{2}\right)^pC_p\left(\int^t_0\sum\limits^\infty_{i=1}\lambda_i^{\gamma-\alpha-2\rho}\left(\int^t_s\left(\frac{s}{r}\right)^{\frac{1}{2}-H}(r-s)^{H-\frac{3}{2}}\mathrm{d}r\right)^2\mathrm{d}s\right)^{\frac{p}{2}}\\
&\le C^p_H\left(H-\frac{1}{2}\right)^pC_p\left(\int^t_0\sum\limits^\infty_{i=1}\lambda_i^{\gamma-\alpha-2\rho}\left(\frac{s}{t}\right)^{1-2H}\left(\int^t_s(r-s)^{H-\frac{3}{2}}\mathrm{d}r\right)^2\mathrm{d}s\right)^{\frac{p}{2}}\\
&\le C^p_H\left(\frac{1}{2-2H}\right)^pC_p\left(t^{2H}\sum\limits^\infty_{i=1}i^{\frac{2(\gamma-\alpha-2\rho)}{d}}\right)^{\frac{p}{2}}.
 \end{split}
 \end{equation}
When $\frac{2(\gamma-\alpha-2\rho)}{d}<-1$, then the infinite series $\sum\limits^\infty_{i=1}i^{\frac{2(\gamma-\alpha-2\rho)}{d}}<\infty$.

One can obtain the following regularity results of the mild solution $u(t)$ and $v(t)$ by using the above estimates, Lemma \ref{le:01}, and \eqref{eq:3.4}.
\begin{theorem}\label{th:1}
Suppose that Assumptions \ref{as:2.1}-\ref{as:2.2} are satisfied, $\left\|u_0\right\|_{L^p(D,\dot{U}^\gamma)}<\infty$, $\left\|v_0\right\|_{L^p(D,\dot{U}^{\gamma-\alpha})}<\infty$, $\varepsilon>0$, $\gamma=\alpha+2\rho-\frac{d+\varepsilon}{2}$, and $\gamma>0$.
Then there exists a unique mild solution $X(t)$ for \eqref{eq:3.2} and
\begin{equation}\label{eq:3.6}
\left\|u(t)\right\|_{L^p(D,\dot{U}^\gamma)}+\left\|v(t)\right\|_{L^p(D,\dot{U}^{\gamma-\alpha})}\lesssim \frac{t^H}{\varepsilon}+\left\|u_0\right\|_{L^p(D,\dot{U}^\gamma)}+\left\|v_0\right\|_{L^p(D,\dot{U}^{\gamma-\alpha})}.
\end{equation}
Furthermore,

$\mathrm{(i)}$ for $\gamma\le\alpha$,
\begin{equation*}
\left\|u(t)-u(s)\right\|_{L^2(D,U)}\lesssim (t-s)^{\frac{\gamma}{\alpha}}\left(\frac{t^H}{\varepsilon}+\left\|u_0\right\|_{L^2(D,\dot{U}^\gamma)}+\left\|v_0\right\|_{L^2(D,\dot{U}^{\gamma-\alpha})}\right);
\end{equation*}

$\mathrm{(ii)}$ for $\gamma>\alpha$,
\begin{equation*}
\left\|u(t)-u(s)\right\|_{L^2(D,U)}\lesssim (t-s)\left(t^H+\left\|u_0\right\|_{L^2(D,\dot{U}^\gamma)}+\left\|v_0\right\|_{L^2(D,\dot{U}^{\gamma-\alpha})}\right).
\end{equation*}

\end{theorem}
\begin{proof}
Let us start with the estimate of $A^{\frac{\gamma}{2}}u(t)$ in $L^p(D,U)$ norm. Combining the triangle inequality, \eqref{eq:3.5}, \eqref{eq:3.5-1}, the expression of $u(t)$ in \eqref{eq:3.4}, and the assumption of $f$, we obtain
\begin{eqnarray*}
&&\left\|A^{\frac{\gamma}{2}}u(t)\right\|_{L^p(D,U)}\\
&&\lesssim\left\|A^{\frac{\gamma}{2}}\cos\left(A^{\frac{\alpha}{2}}t\right)u_0\right\|_{L^p(D,U)}+\left\|A^{\frac{\gamma-\alpha}{2}}\sin\left(A^{\frac{\alpha}{2}}t\right)v_0\right\|_{L^p(D,U)}\\
&&~~~+\left\|\int^t_0A^{\frac{\gamma-\alpha}{2}}\sin\left(A^{\frac{\alpha}{2}}(t-s)\right)f\left(u(s)\right)\mathrm{d}s\right\|_{L^p(D,U)}\\
&&~~~+\left\|\int^t_0A^{\frac{\gamma-\alpha}{2}}\sin\left(A^{\frac{\alpha}{2}}(t-s)\right)\mathrm{d}B_H(s)\right\|_{L^p(D,U)}\\
&&\lesssim\left\|A^{\frac{\gamma}{2}}u_0\right\|_{L^p(D,U)}+\left\|A^{\frac{\gamma-\alpha}{2}}v_0\right\|_{L^p(D,U)}+\int^t_0\left\|A^{\frac{\gamma}{2}}u(s)\right\|_{L^p(D,U)}\mathrm{d}s+\frac{t^H}{\varepsilon}.
\end{eqnarray*}
The application of Gr\"onwall's inequality leads to
\begin{equation*}
\left\|A^{\frac{\gamma}{2}}u(t)\right\|_{L^p(D,U)}
\lesssim\left\|A^{\frac{\gamma}{2}}u_0\right\|_{L^p(D,U)}+\left\|A^{\frac{\gamma-\alpha}{2}}v_0\right\|_{L^p(D,U)}+\frac{t^H}{\varepsilon}.
\end{equation*}
The bound of $\left\|v(t)\right\|_{L^p(D,\dot{U}^{\gamma-\alpha})}$ can be achieved in the same way, that is
\begin{equation*}
\left\|A^{\frac{\gamma-\alpha}{2}}v(t)\right\|_{L^p(D,U)}
\lesssim\left\|A^{\frac{\gamma}{2}}u_0\right\|_{L^p(D,U)}+\left\|A^{\frac{\gamma-\alpha}{2}}v_0\right\|_{L^p(D,U)}+\frac{t^H}{\varepsilon}.
\end{equation*}
Then using above estimates leads to
\begin{equation*}
\left\|u(t)\right\|_{L^p(D,\dot{U}^\gamma)}+\left\|v(t)\right\|_{L^p(D,\dot{U}^{\gamma-\alpha})}\lesssim \frac{t^H}{\varepsilon}+\left\|u_0\right\|_{L^p(D,\dot{U}^\gamma)}+\left\|v_0\right\|_{L^p(D,\dot{U}^{\gamma-\alpha})}.
\end{equation*}
Next, we discuss the time H\"older continuity of $u(t)$. 
Equation \eqref{eq:3.4} implies
\begin{eqnarray*}
&\mathrm{E}&\left[\left\|u(t)-u(s)\right\|^2\right]\\
&\lesssim&\mathrm{E}\left[\left\|\left(\cos\left(A^{\frac{\alpha}{2}}t\right)-\cos\left(A^{\frac{\alpha}{2}}s\right)\right)u_0\right\|^2+\left\|A^{-\frac{\alpha}{2}}\left(\sin\left(A^{\frac{\alpha}{2}}t\right)-\sin\left(A^{\frac{\alpha}{2}}s\right)\right)v_0\right\|^2\right]\\
&+&\mathrm{E}\left[\left\|\int^t_sA^{-\frac{\alpha}{2}}\sin\left(A^{\frac{\alpha}{2}}(t-r)\right)f\left(u(r)\right)\mathrm{d}r\right\|^2+\left\|\int^t_sA^{-\frac{\alpha}{2}}\sin\left(A^{\frac{\alpha}{2}}(t-r)\right)\mathrm{d}B_H(r)\right\|^2\right]\\
&+&\mathrm{E}\left[\left\|\int^s_0A^{-\frac{\alpha}{2}}\left(\sin\left(A^{\frac{\alpha}{2}}(t-r)\right)-\sin\left(A^{\frac{\alpha}{2}}(s-r)\right)\right)f\left(u(r)\right)\mathrm{d}r\right\|^2\right]\\
&+&\mathrm{E}\left[\left\|\int^s_0A^{-\frac{\alpha}{2}}\left(\sin\left(A^{\frac{\alpha}{2}}(t-r)\right)-\sin\left(A^{\frac{\alpha}{2}}(s-r)\right)\right)\mathrm{d}B_H(r)\right\|^2\right]\\
&\lesssim&I_1+I_2+I_3+I_4+I_5+I_6.
\end{eqnarray*}
For $\frac{1}{2}<H<1$, the inequality $\cos(a)-\cos(b)\lesssim |a-b|^\theta \ (0\le\theta\le1)$ implies
\begin{eqnarray*}
I_1&=&\mathrm{E}\left[\left\|\sum_i\left(\cos\left(\lambda_i^{\frac{\alpha}{2}}t\right)-\cos\left(\lambda_i^{\frac{\alpha}{2}}s\right)\right)\left\langle u_0,\phi_{i}(x)\right\rangle\phi_{i}(x)\right\|^2\right]\\
&\lesssim&\mathrm{E}\left[\sum_i\lambda_i^{\gamma}(t-s)^{2\cdot\min\{\frac{\gamma}{\alpha},1\}}\left\langle u_0,\phi_{i}(x)\right\rangle^2\right]\\
&\lesssim&(t-s)^{\min\{\frac{2\gamma}{\alpha},2\}}\mathrm{E}\left[\left\|A^{\frac{\gamma}{2}}u_0\right\|^2\right].
\end{eqnarray*}
Similar to the derivation of $I_1$, it holds that
\begin{eqnarray*}
I_2\lesssim (t-s)^{\min\{\frac{2\gamma}{\alpha},2\}}\mathrm{E}\left[\left\|A^{\frac{\gamma-\alpha}{2}}v_0\right\|^2\right].
\end{eqnarray*}
For the H\"older regularity of the third term, using Equation \eqref{eq:3.6} and Assumption \ref{as:2.1} leads to
\begin{eqnarray*}
I_3&\lesssim& (t-s)\int^t_s\mathrm{E}\left[\left\|A^{-\frac{\alpha}{2}}\sin\left(A^{\frac{\alpha}{2}}(t-r)\right)f\left(u(r)\right)\right\|^2\right]\mathrm{d}r\\
&\lesssim& (t-s)\int^t_s\mathrm{E}\left[\left\|u(r)\right\|^2+1\right]\mathrm{d}r\\
&\lesssim& (t-s)^2\left(\mathrm{E}\left[\left\|A^{\frac{\gamma}{2}}u_0\right\|^2\right]+\mathrm{E}\left[\left\|A^{\frac{\gamma-\alpha}{2}}v_0\right\|^2\right]+1\right).
\end{eqnarray*}
Take $\theta_1=\min\left\{\frac{\gamma}{2\alpha},\frac{1}{2}\right\}$. Combining the fact that $|\sin(t)|\lesssim|t|^{\theta_1}$ $(t\in\mathbb{R})$ and Lemma \ref{le:02} leads to
\begin{eqnarray*}
I_4&=& \mathrm{E}\left[\left\|\int^t_sA^{-\frac{\alpha}{2}}\sin\left(A^{\frac{\alpha}{2}}(t-r)\right)\mathrm{d}B_H(r)\right\|^2\right]\\
&\lesssim& \sum_i\lambda_i ^{-\alpha-2\rho}\int^t_s\int^t_s\left|\sin\left(\lambda_i^{\frac{\alpha}{2}}(t-r)\right)\right|\times\left|\sin\left(\lambda_i^{\frac{\alpha}{2}}(t-r_1)\right)\right|\times|r-r_1|^{2H-2}\mathrm{d}r\mathrm{d}r_1\\
&\lesssim& (t-s)^{2H+\min\left\{\frac{\gamma}{\alpha},1\right\}}\sum_i\lambda_i ^{\frac{\gamma}{2}-\alpha-2\rho}\\
&\lesssim& (t-s)^{2H+\min\left\{\frac{\gamma}{\alpha},1\right\}}.
\end{eqnarray*}
Using $\sin(a)-\sin(b)\lesssim |a-b|^{\theta_2}$ and the assumption of $f$ with $\theta_2=\min\left\{\frac{\gamma}{\alpha},1\right\}$, we have
\begin{eqnarray*}
I_5&\lesssim&(t-s)^{\min\left\{\frac{2\gamma}{\alpha},2\right\}}\mathrm{E}\left[\left\|\int^s_0A^{\frac{\gamma-\alpha}{2}}f\left(u(r)\right)\mathrm{d}r\right\|^2\right]\\
&\lesssim& (t-s)^{\min\left\{\frac{2\gamma}{\alpha},2\right\}}\left(\mathrm{E}\left[\left\|A^{\frac{\gamma}{2}}u_0\right\|^2\right]+\mathrm{E}\left[\left\|A^{\frac{\gamma-\alpha}{2}}v_0\right\|^2\right]+t^{2H}\right).
\end{eqnarray*}
For $\gamma>\alpha$, by using Lemma \ref{le:02} and the inequality $\sin\left(\lambda_i^{\frac{\alpha}{2}}(t-r)\right)-\sin\left(\lambda_i^{\frac{\alpha}{2}}(s-r)\right)\lesssim\lambda_i^{\frac{\alpha}{2}}(t-s)$, we get the bound of $I_6$, i.e.,
\begin{eqnarray*}
I_6&\lesssim&\sum_i\lambda_i ^{-\alpha-2\rho}\int^s_0\int^s_0\left|\sin\left(\lambda_i^{\frac{\alpha}{2}}(t-r)\right)-\sin\left(\lambda_i^{\frac{\alpha}{2}}(s-r)\right)\right|\\
&&\times\left|\sin\left(\lambda_i^{\frac{\alpha}{2}}(t-r_1)\right)-\sin\left(\lambda_i^{\frac{\alpha}{2}}(s-r_1)\right)\right|\times|r-r_1|^{2H-2}\mathrm{d}r\mathrm{d}r_1\\
&\lesssim& t^{2H}(t-s)^2\sum_i\lambda_i ^{\alpha-\gamma+\gamma-\alpha-2\rho}\\
&\lesssim& t^{2H}(t-s)^2.
\end{eqnarray*}
When $\gamma\le\alpha$, using $\sin\left(\lambda_i^{\frac{\alpha}{2}}(t-r)\right)-\sin\left(\lambda_i^{\frac{\alpha}{2}}(s-r)\right)\lesssim\lambda_i^{\frac{\gamma}{2}}(t-s)^{\frac{\gamma}{\alpha}}$ leads to
\begin{eqnarray*}
I_6&\lesssim& t^{2H}(t-s)^{\frac{2\gamma}{\alpha}}\sum_i\lambda_i ^{\gamma-\alpha-2\rho}\\
&\lesssim& \frac{t^{2H}}{\varepsilon}(t-s)^{\frac{2\gamma}{\alpha}}.
\end{eqnarray*}
Then collecting the above estimates arrives at
\begin{equation*}
\mathrm{E}\left[\left\|u(t)-u(s)\right\|^2\right]\lesssim (t-s)^2\left(t^{2H}+\mathrm{E}\left[\left\|A^{\frac{\gamma}{2}}u_0\right\|^2\right]+\mathrm{E}\left[\left\|A^{\frac{\gamma-\alpha}{2}}v_0\right\|^2\right]\right),\ \gamma>\alpha
\end{equation*}
and
\begin{equation*}
\mathrm{E}\left[\left\|u(t)-u(s)\right\|^2\right]\lesssim (t-s)^{\frac{2\gamma}{\alpha}}\left(\frac{t^{2H}}{\varepsilon}+\mathrm{E}\left[\left\|A^{\frac{\gamma}{2}}u_0\right\|^2\right]+\mathrm{E}\left[\left\|A^{\frac{\gamma-\alpha}{2}}v_0\right\|^2\right]\right),\ \gamma\le\alpha.
\end{equation*}
When $H=\frac{1}{2}$, the above H\"older regularity results still hold. The proof is completed.
\end{proof}

In fact, one can get an equivalent form of \eqref{eq:3.2} by using variable substitution, 
the regularity of whose solution is better than the one of the solution of \eqref{eq:3.2}.
Let
\begin{equation}\label{eq:3.6-1}
Z(t)=X(t)-\int^t_0\mathrm{e}^{\Lambda(t-s)}\left[\begin{array}{c}
0\\
I
\end{array}\right]\mathrm{d}B_H(s),
\end{equation}
where
\begin{equation*}
Z(t)=
\left[\begin{array}{c}
z(t)\\[1.5mm]
\dot{z}(t)
\end{array} \right].
\end{equation*}
If $X(t)$ is the unique mild solution of \eqref{eq:3.2}, then $Z(t)$ is the unique mild solution of the  partial differential equation
\begin{equation}\label{eq:3.7}
\frac{\mathrm{d}}{\mathrm{d}t}Z(t)=\Lambda Z(t)+\left[\begin{array}{c}
0\\
f(u(t))
\end{array} \right] ~ {\rm for } ~ t\in(0,T] ~ {\rm with } ~ Z(0)=X(0),
\end{equation}
where
\begin{equation*}
f\left(u(t)\right)=f\left(z(t)+\int^{t}_0A^{-\frac{\alpha}{2}}\sin\left(A^{\frac{\alpha}{2}}(t-s)\right)\mathrm{d}B_H(s)\right).
\end{equation*}
The unique mild solution of \eqref{eq:3.7} is given by
\begin{equation}\label{eq:3.7-1}
Z(t)=\mathrm{e}^{\Lambda t}Z(0)+\int^t_0\mathrm{e}^{\Lambda (t-s)}\left[\begin{array}{c}
0\\
f(u(t))
\end{array} \right]\mathrm{d}s.
\end{equation}
Then we can obtain $z(t)$ and $\dot{z}(t)$ as
\begin{equation}\label{eq:3.7-2}
\begin{split}
z(t)&=\cos\left(A^{\frac{\alpha}{2}}t\right)u_0+A^{-\frac{\alpha}{2}}\sin\left(A^{\frac{\alpha}{2}}t\right)v_0+\int^t_0A^{-\frac{\alpha}{2}}\sin\left(A^{\frac{\alpha}{2}}(t-s)\right)f\left(u(s)\right)\mathrm{d}s,\\
\dot{z}(t)&=-A^{\frac{\alpha}{2}}\sin\left(A^{\frac{\alpha}{2}}t\right)u_0+\cos\left(A^{\frac{\alpha}{2}}t\right)v_0+\int^t_0\cos\left(A^{\frac{\alpha}{2}}(t-s)\right)f\left(u(s)\right)\mathrm{d}s.
\end{split}
\end{equation}

Equation \eqref{eq:3.7} will be used to obtain the spatial semi-discretization solution of \eqref{eq:3.2}. Therefore we give the following estimates, which will be used to discuss the spatial error.
\begin{corollary}\label{co:2}
Suppose that Assumptions \ref{as:2.1}-\ref{as:2.2} are satisfied, $\left\|u_0\right\|_{L^p(D,\dot{U}^{\gamma+\alpha})}<\infty$, $\left\|v_0\right\|_{L^p(D,\dot{U}^{\gamma})}<\infty$, $\varepsilon>0$, $\gamma=\alpha+2\rho-\frac{d+\varepsilon}{2}$, and $\gamma>0$.
Then there exists a unique mild solution $Z(t)$ for \eqref{eq:3.7} and
\begin{equation*}
\left\|z(t)\right\|_{L^p(D,\dot{U}^{\gamma+\alpha})}+\left\|\dot{z}(t)\right\|_{L^p(D,\dot{U}^{\gamma})}\lesssim \frac{t^H}{\varepsilon}+\left\|u_0\right\|_{L^p(D,\dot{U}^{\gamma+\alpha})}+\left\|v_0\right\|_{L^p(D,\dot{U}^{\gamma})}.
\end{equation*}
\end{corollary}
The proof of this corollary is very similar to the proof of Theorem \ref{th:1}.

\section{Galerkin approximation for spatial discretization} \label{sec:4}

The convergence rate of the spectral approximation depends on the regularity of the mild solution in space. Theorem \ref{th:1} and Corollary \ref{co:2} show that $z(t)$ is more regular than $u(t)$ in space; so we obtain the spatial approximation of \eqref{eq:3.2} by using the spectral Galerkin method to discretize \eqref{eq:3.7} and postprocessing the stochastic integral.

A finite dimensional subspace of $U$ will be needed to implement the Galerkin spatial approximation of \eqref{eq:3.7}. Denoting the $N$ dimensional subspace of $U$ by $U^N$, the sequence $\left\{\phi_{1}(x), \dots, \phi_{i}(x), \dots, \phi_{N}(x)\right\}_{N\in\mathbb{N}}$ is an orthonormal basis of $U^N$. Then we introduce the projection operator $P_N:\,U\to U^N$, for $\xi\in U$,
\begin{equation*}
P_N\xi=\sum^N_{i=1}\left\langle\xi,\phi_{i}(x) \right\rangle\phi_{i}(x)
\end{equation*}
and
\begin{equation}\label{eq:4.1}
\left\langle P_N\xi,\chi\right\rangle=\left\langle \xi,\chi\right\rangle \quad \forall \chi\in U^N.
\end{equation}
To obtain the Galerkin formulation of \eqref{eq:3.7}, we look for $z^N(t)\in U^N$ and $\dot{z}^N\in U^N$ such that
\begin{eqnarray}\label{eq:4.3}
\left[\begin{array}{l}
\left\langle\mathrm{d}z^N(t),\chi\right\rangle\\[1.5mm]
\left\langle\mathrm{d}\dot{z}^N(t),\chi\right\rangle
\end{array} \right]&=&\Lambda\left[\begin{array}{cc}
\left\langle z^N(t),\chi\right\rangle\\[1.5mm]
\left\langle \dot{z}^N(t),\chi\right\rangle
\end{array} \right]\mathrm{d}t+\left[\begin{array}{c}
0\\[1.5mm]
\left\langle f\left(u^N(t)\right),\chi\right\rangle
\end{array} \right]\mathrm{d}t
\end{eqnarray}
and
\begin{equation*}
\left\langle z^N_0,\chi\right\rangle =\left\langle z_0,\chi\right\rangle,\ \left\langle \dot{z}^N_0,\chi\right\rangle =\left\langle \dot{z}_0,\chi\right\rangle,
\end{equation*}
where
\begin{equation*}
f\left(u^N(t)\right)=f\left(z^N(t)+P_N\int^{t}_0A^{-\frac{\alpha}{2}}\sin\left(A^{\frac{\alpha}{2}}(t-s)\right)\mathrm{d}B_H(s)\right).
\end{equation*}
The Galerkin formulation of \eqref{eq:3.7} is obtained by using \eqref{eq:4.1} and \eqref{eq:4.3}, that is
\begin{eqnarray}\label{eq:4.4}
\left[\begin{array}{l}
\mathrm{d}z^N(t)\\[1.5mm]
\mathrm{d}\dot{z}^N(t)
\end{array} \right]=\Lambda\left[\begin{array}{cc}
 z^N(t)\\[1.5mm]
 \dot{z}^N(t)
\end{array} \right]\mathrm{d}t+\left[\begin{array}{c}
0\\[1.5mm]
 f_N\left(u^N(t)\right)
\end{array} \right]\mathrm{d}t
\end{eqnarray}
and
\begin{equation*}
z^N_0=P_Nu_0,\ \dot{z}^N_0=P_Nv_0,
\end{equation*}
where $f_N=P_Nf$. We can obtain the mild solution of \eqref{eq:4.4} as
\begin{equation}\label{eq:4.4-11}
\begin{split}
z^N(t)=&\cos\left(A^{\frac{\alpha}{2}}t\right)u^N_0+A^{-\frac{\alpha}{2}}\sin\left(A^{\frac{\alpha}{2}}t\right)v^N_0+\int^t_0A^{-\frac{\alpha}{2}}\sin\left(A^{\frac{\alpha}{2}}(t-s)\right)f_N\left(u^N(s)\right)\mathrm{d}s,\\
\dot{z}^N(t)=&-A^{\frac{\alpha}{2}}\sin\left(A^{\frac{\alpha}{2}}t\right)u^N_0+\cos\left(A^{\frac{\alpha}{2}}t\right)v^N_0+\int^t_0\cos\left(A^{\frac{\alpha}{2}}(t-s)\right)f_N\left(u^N(s)\right)\mathrm{d}s.
\end{split}
\end{equation}
Then the spatial semi-discretization solution of \eqref{eq:3.2} is given by
\begin{eqnarray}\label{eq:4.4-1}
u^N(t)=z^N(t)+P_{N}\int^{t}_0A^{-\frac{\alpha}{2}}\sin\left(A^{\frac{\alpha}{2}}(t-s)\right)\mathrm{d}B_H(s)
\end{eqnarray}
and
\begin{eqnarray}\label{eq:4.4-2}
v^N(t)&=\dot{z}^N(t)+P_{N}\int^{t}_0\cos\left(A^{\frac{\alpha}{2}}(t-s)\right)\mathrm{d}B_H(s).
\end{eqnarray}

In fact, Corollary \ref{co:2} shows $z(t)$ has better regularity
than the stochastic integral $\int^{t}_0A^{-\frac{\alpha}{2}}\sin\left(A^{\frac{\alpha}{2}}(t-s)\right)\mathrm{d}B_H(s)$ in space; so we can improve accuracy of the Galerkin approximate solution by postprocessing the stochastic integral of \eqref{eq:4.4-1}.
Let $N_1=\left[N^{\theta_3}\right]$ and $\theta_3\ge1$, with $[y]$ being the nearest integer to $y$. Then the spatial semi-discretization solution of \eqref{eq:3.2} can be expressed as
\begin{eqnarray}\label{eq:4.5}
u^N(t)&=&z^N(t)+P_{N_1}\int^{t}_0A^{-\frac{\alpha}{2}}\sin\left(A^{\frac{\alpha}{2}}(t-s)\right)\mathrm{d}B_H(s)
\end{eqnarray}
and
\begin{eqnarray}\label{eq:4.6}
v^N(t)=\dot{z}^N(t)+P_{N_1}\int^{t}_0\cos\left(A^{\frac{\alpha}{2}}(t-s)\right)\mathrm{d}B_H(s).
\end{eqnarray}
Using \eqref{eq:4.4} leads to
\begin{eqnarray}\label{eq:4.6-1}
\dot{z}^N(t)&=&-A^{\frac{\alpha}{2}}\sin\left(A^{\frac{\alpha}{2}}(t-s)\right)z^N(s)+\cos\left(A^{\frac{\alpha}{2}}(t-s)\right)\dot{z}^N(s)\\
&~&+\int^t_s\cos\left(A^{\frac{\alpha}{2}}(t-r)\right)f_N\left(u^N(r)\right)\mathrm{d}r.\nonumber
\end{eqnarray}
Then substituting \eqref{eq:4.5} and \eqref{eq:4.6} into \eqref{eq:4.6-1} leads to
\begin{eqnarray}\label{eq:4.7-1}
v^N(t)
&=&-A^{\frac{\alpha}{2}}\sin\left(A^{\frac{\alpha}{2}}(t-s)\right)u^N(s)+\cos\left(A^{\frac{\alpha}{2}}(t-s)\right)v^N(s)\\
&&+P_{N_1}\int^{t}_s\cos\left(A^{\frac{\alpha}{2}}(t-r)\right)\mathrm{d}B_H(r)\nonumber\\
&&+\int^{t}_{s}\cos\left(A^{\frac{\alpha}{2}}(t-r)\right)f_N\left(u^N(r)\right)\mathrm{d}r,\nonumber
\end{eqnarray}
which will be used to prove Theorem \ref{th:7}.

Combining Theorem \ref{th:1}, \eqref{eq:4.5} and \eqref{eq:4.6}, we now deduce the following regularity results of the spatial semi-discretization solution.
\begin{corollary}\label{co:3}
Suppose that Assumption \ref{as:2.1}-\ref{as:2.2} are satisfied; $\left\|u_0\right\|_{L^p(D,\dot{U}^{\gamma})}<\infty$, $\left\|v_0\right\|_{L^p(D,\dot{U}^{\gamma-\alpha})}<\infty$, $\varepsilon>0$, $\gamma=\alpha+2\rho-\frac{d+\varepsilon}{2}$ and $\gamma>0$.
The approximate solution $u^N(t)$ and $v^N(t)$ is expressed by \eqref{eq:4.5} and \eqref{eq:4.6}, respectively. Then
\begin{equation*}
\left\|u^N(t)\right\|_{L^p(D,\dot{U}^{\gamma})}+\left\|v^N(t)\right\|_{L^p(D,\dot{U}^{\gamma-\alpha})}\lesssim \frac{t^H}{\varepsilon}+\left\|u_0\right\|_{L^p(D,\dot{U}^{\gamma})}+\left\|v_0\right\|_{L^p(D,\dot{U}^{\gamma-\alpha})}
\end{equation*}
and

$\mathrm{(i)}$ for $\gamma\le\alpha$,
\begin{equation*}
\left\|u^N(t)-u^N(s)\right\|_{L^2(D,U)}\lesssim (t-s)^{\frac{\gamma}{\alpha}}\left(\frac{t^H}{\varepsilon}+\left\|u_0\right\|_{L^2(D,\dot{U}^\gamma)}+\left\|v_0\right\|_{L^2(D,\dot{U}^{\gamma-\alpha})}\right).
\end{equation*}

$\mathrm{(ii)}$ for $\gamma>\alpha$,
\begin{equation*}
\left\|u^N(t)-u^N(s)\right\|_{L^2(D,U)}\lesssim (t-s)\left(t^H+\left\|u_0\right\|_{L^2(D,\dot{U}^\gamma)}+\left\|v_0\right\|_{L^2(D,\dot{U}^{\gamma-\alpha})}\right).
\end{equation*}

\end{corollary}
The proof of this corollary is done in the same way as Theorem \ref{th:1}.

Next, the following lemma are given to analyze the error of the approximate solution $u^N(t)$ in \eqref{eq:4.5}.
\begin{lemma}\label{le:4}
If $\mathrm{E}\left[\|A^{\frac{\nu}{2}}\xi\|^2\right]<\infty$, $\xi\in U$, then
\begin{equation*}
\mathrm{E}\left[\|(P_N-I)\xi\|^2\right]\lesssim \lambda_{N+1}^{-\nu}\mathrm{E}\left[\|A^{\frac{\nu}{2}}\xi\|^2\right].
\end{equation*}
\end{lemma}

Corollary \ref{co:2} shows $z(t)\in L^p(D,\dot{U}^{\gamma+\alpha})$.
From Lemma \ref{le:4}, we can infer that $\mathrm{E}\left[\|z(t)-P_Nz(t)\|^2\right]\lesssim \lambda_{N+1}^{-\gamma-\alpha}\mathrm{E}\left[\|A^{\frac{\gamma+\alpha}{2}}z(t)\|^2\right]$. Therefore we can obtain an order of $\frac{\gamma+\alpha}{d}$ for the spatial semi-discretization solution by adjusting $N_1$ in \eqref{eq:4.5}. Let $N_1=\left[N^{\frac{\gamma+\alpha}{\gamma}}\right]$. Then we get following result.
\begin{theorem} \label{th:5}
Let $X(t)$ and $Z^N(t)$ be the mild solution of \eqref{eq:3.2} and \eqref{eq:3.7}, respectively. Suppose that Assumption \ref{as:2.1}-\ref{as:2.2} are satisfied. Let $\left\|u_0\right\|_{L^p(D,\dot{U}^{\gamma+\alpha})}<\infty$, $\left\|v_0\right\|_{L^p(D,\dot{U}^{\gamma})}<\infty$, $\varepsilon>0$, $\gamma=\alpha+2\rho-\frac{d+\varepsilon}{2}$, $\gamma>0$; and \eqref{eq:4.5} is the approximation of $u(t)$. If $N_1=\left[N^{\frac{\gamma+\alpha}{\gamma}}\right]$, then we have
\begin{equation*}
\left\|u(t)-u^N(t)\right\|_{L^2(D,U)}\lesssim N^{-\frac{\gamma+\alpha}{d}}\left(\frac{t^H}{\varepsilon}+\left\|u_0\right\|_{L^2(D,\dot{U}^{\gamma+\alpha})}+\left\|v_0\right\|_{L^2(D,\dot{U}^{\gamma})}\right).
\end{equation*}
\end{theorem}
\begin{proof}
In the first place, using the triangle inequality, Corollary \ref{co:2}, Lemma \ref{le:4} and \eqref{eq:4.5}, we obtain
\begin{equation}\label{eq:4.7}
\begin{split}
&\left\|u(t)-u^N(t)\right\|_{L^2(D,U)}\\
& \lesssim \left\|z(t)-z^N(t)\right\|_{L^2(D,U)}\\
&  ~~~+\left\|\int^{t}_0\sum^{\infty}_{i=N_1+1}\lambda_i^{-\frac{\alpha}{2}}\sin\left(\lambda_i^{\frac{\alpha}{2}}(t_m-s)\right)\sigma_i\phi_i(x)\mathrm{d}\beta^i_H(s)\right\|_{L^2(D,U)}\\
&\lesssim\left\|z(t)-P_Nz(t)\right\|_{L^2(D,U)}+\left\|P_Nz(t)-z^N(t)\right\|_{L^2(D,U)}+\frac{t^H}{\varepsilon}N_1^{-\frac{\gamma}{d}}\\
&\lesssim\left\|P_Nz(t)-z^N(t)\right\|_{L^2(D,U)}+N^{-\frac{\gamma+\alpha}{d}}\left(\frac{t^H}{\varepsilon}+\left\|u_0\right\|_{L^2(D,\dot{U}^{\gamma+\alpha})}+\left\|v_0\right\|_{L^2(D,\dot{U}^{\gamma})}\right).
\end{split}
\end{equation}
Then we need to estimate the bound of $\left\|P_Nz(t)-z^N(t)\right\|_{L^2(D,U)}$. The definition of projection operator $P_N$ implies that $P_NA^{-\frac{\alpha}{2}}\sin\left(A^{\frac{\alpha}{2}}t\right)f=A^{-\frac{\alpha}{2}}\sin\left(A^{\frac{\alpha}{2}}t\right)P_Nf.$
Thus, first performing $P_N$ on \eqref{eq:3.7-2} and then doing subtraction with respect to \eqref{eq:4.4-11} leads to
\begin{eqnarray*}
&&\left\|P_Nz(t)-z^N(t)\right\|_{L^2(D,U)}\\
&&=\left\|\int^{t}_{0}A^{-\frac{\alpha}{2}}\sin\left(A^{\frac{\alpha}{2}}(t_{m+1}-s)\right)\left(f_N\left(u(s)\right)-f_N\left(u^N(s)\right)\right)\mathrm{d}s\right\|_{L^2(D,U)}\\
&&\lesssim\int^{t}_{0}\left\|u(s)-u^N(s)\right\|_{L^2(D,U)}\mathrm{d}s.
\end{eqnarray*}
Then the above estimates and the Gr\"onwall inequality imlpy
\begin{eqnarray}\label{eq:4.8}
\left\|u(t)-u^N(t)\right\|_{L^2(D,U)}&\lesssim & N^{-\frac{\gamma+\alpha}{d}}\left(\frac{t^H}{\varepsilon}+\left\|u_0\right\|_{L^2(D,\dot{U}^{\gamma+\alpha})}+\left\|v_0\right\|_{L^2(D,\dot{U}^{\gamma})}\right).
\end{eqnarray}
If $N>1$, choosing $\varepsilon=\frac{1}{\log(N)}$, we have
\begin{eqnarray*}
\left\|u(t)-u^N(t)\right\|_{L^2(D,U)}&\lesssim & N^{-\frac{2\alpha+4\rho+2\alpha-d-\varepsilon}{2d}}\left(\frac{t^H}{\varepsilon}+\left\|u_0\right\|_{L^2(D,\dot{U}^{\gamma+\alpha})}+\left\|v_0\right\|_{L^2(D,\dot{U}^{\gamma})}\right)\\
&\lesssim & N^{-\frac{4\alpha+4\rho-d}{2d}}\left(t^H\log(N)+\left\|u_0\right\|_{L^2(D,\dot{U}^{\gamma+\alpha})}+\left\|v_0\right\|_{L^2(D,\dot{U}^{\gamma})}\right).
\end{eqnarray*}
\end{proof}

\section{Fully discrete scheme}\label{sec:5}
In this section, we concern the time discretization of \eqref{eq:4.4}. Meanwhile the error estimates of the fully discrete scheme are derived.

Let $z^N_{m}$ and $\bar{z}^N_{m}$ denote respectively the approximation of $z^N(t_{m})$ and $\dot{z}^N(t_{m})$ with
fixed time step size $\tau=\frac{T}{M}$ and $t_m=m\tau$ $(m=0,1,2,\dots,M)$. Using the stochastic trigonometric method, we can get the full discrete scheme of \eqref{eq:3.2}
\begin{eqnarray}\label{eq:5.1}
\left[\begin{array}{c}
z^N_{m+1}\\[1.5mm]
\bar{z}^N_{m+1}
\end{array} \right]&=&\left[\begin{array}{cc}
\cos\left(A^{\frac{\alpha}{2}}\tau\right)&A^{-\frac{\alpha}{2}}\sin\left(A^{\frac{\alpha}{2}}\tau\right)\\[1.5mm]
-A^{\frac{\alpha}{2}}\sin\left(A^{\frac{\alpha}{2}}\tau\right)&\cos\left(A^{\frac{\alpha}{2}}\tau\right)
\end{array} \right]\left[\begin{array}{c}z^N_{m}\\[1.5mm]
\bar{z}^N_{m}
\end{array} \right]\\
&+&\tau\left[\begin{array}{c}A^{-\frac{\alpha}{2}}\sin\left(A^{\frac{\alpha}{2}}\tau\right)f_N\left(u^N_{m}\right)\nonumber\\[1.5mm]
\cos\left(A^{\frac{\alpha}{2}}\tau\right)f_N\left(u^N_{m}\right)
\end{array} \right],
\end{eqnarray}
where
\begin{equation*}
u^N_m=z^N_m+P_{N_1}\int^{t_{m}}_0A^{-\frac{\alpha}{2}}\sin\left(A^{\frac{\alpha}{2}}(t_{m}-s)\right)\mathrm{d}B_H(s).
\end{equation*}
By using recursion form of \eqref{eq:5.1}, we get
\begin{eqnarray}\label{eq:5.1-1}
\left[\begin{array}{c}
z^N_{m+1}\\[1.5mm]
\bar{z}^N_{m+1}
\end{array} \right]&=&\left[\begin{array}{cc}
\cos\left(A^{\frac{\alpha}{2}}t_{m+1}\right)&A^{-\frac{\alpha}{2}}\sin\left(A^{\frac{\alpha}{2}}t_{m+1}\right)\\[1.5mm]
-A^{\frac{\alpha}{2}}\sin\left(A^{\frac{\alpha}{2}}t_{m+1}\right)&\cos\left(A^{\frac{\alpha}{2}}t_{m+1}\right)
\end{array} \right]\left[\begin{array}{c}u^N_{0}\\[1.5mm]
v^N_{0}
\end{array} \right]\\
&&+\tau\sum^{m}_{j=0}\left[\begin{array}{c}A^{-\frac{\alpha}{2}}\sin\left(A^{\frac{\alpha}{2}}\left(t_{m+1}-t_j\right)\right)f_N\left(u^N_{j}\right)\\[1.5mm]
\cos\left(A^{\frac{\alpha}{2}}\left(t_{m+1}-t_j\right)\right)f_N\left(u^N_{j}\right)
\end{array} \right].\nonumber
\end{eqnarray}
Then we get the approximations of $u(t)$ and $v(t)$, that is,
\begin{eqnarray}\label{eq:5.2}
u^N_{m+1}=z^N_{m+1}+P_{N_1}\int^{t_{m+1}}_0A^{-\frac{\alpha}{2}}\sin\left(A^{\frac{\alpha}{2}}(t_{m+1}-s)\right)\mathrm{d}B_H(s)
\end{eqnarray}
and
\begin{eqnarray}\label{eq:5.2-1}
v^N_{m+1}=\bar{z}^N_{m+1}+P_{N_1}\int^{t_{m+1}}_0\cos\left(A^{\frac{\alpha}{2}}(t_{m+1}-s)\right)\mathrm{d}B_H(s).
\end{eqnarray}
As $\frac{1}{2}<H<1$, although $\int^{t_{m+1}}_0\sin\left(A^{\frac{\alpha}{2}}(t_{m+1}-s)\right)\mathrm{d}B_H(s)$ is a Gaussian process, it is difficult to accurately simulate this process; thus we give the approximation of stochastic integral, that is
\begin{eqnarray}\label{eq:5.2-11}
P_{N_1}\sum^m_{j=0}\int^{t_{j+1}}_{t_{j}}A^{-\frac{\alpha}{2}}\sin\left(A^{\frac{\alpha}{2}}(t_{m+1}-t_j)\right)\mathrm{d}B_H(s).
\end{eqnarray}
Using Lemma \ref{le:02}, we obtain the error estimate
\begin{equation*}
\begin{split}
\mathrm{E}&\left[\left\|P_{N_1}\sum^m_{j=0}\int^{t_{j+1}}_{t_{j}}A^{-\frac{\alpha}{2}}\left(\sin\left(A^{\frac{\alpha}{2}}(t_{m+1}-s)\right)-\sin\left(A^{\frac{\alpha}{2}}(t_{m+1}-t_j)\right)\right)\mathrm{d}B_H(s)\right\|^2\right]\\
\lesssim&t^{2H}_{m+1}\tau^{\min\left\{\frac{2\gamma}{\alpha},2\right\}}\sum^{N_1}_{i=0}\lambda_i ^{\min\left\{\gamma,\alpha\right\}-\alpha-2\rho},
\end{split}
\end{equation*}
which implies that this approximation \eqref{eq:5.2-11} does not change  the temporal convergence rate of scheme \eqref{eq:5.1}.
The proof of this estimate is done in the same way as \eqref{eq:5.9-1}. For $H=\frac{1}{2}$, the simulation of stochastic integral is easily implementable without approximation.
We now investigate the error estimates of the fully discrete scheme \eqref{eq:5.1}. The triangle inequality implies that
\begin{eqnarray*} \left\|u(t_m)-u^N_m\right\|_{L^2(D,U)}\lesssim\left\|u(t_m)-u^N(t_m)\right\|_{L^2(D,U)}+\left\|u^N(t_m)-u^N_m\right\|_{L^2(D,U)}.
\end{eqnarray*}
Thus, we need to give the bound estimate of $\left\|u^N(t_m)-u^N_m\right\|_{L^2(D,U)}$. This bound estimate can be obtained by using the time H\"older regularity of $u^N(t)$, \eqref{eq:4.5}, and \eqref{eq:5.2}. Therefore combining the error estimate of the approximation \eqref{eq:5.2-11}, Corollary \ref{co:3}, and Theorem \ref{th:5} leads to the following results.
\begin{proposition} \label{Pro:6}
Let $u(t_{m+1})$ and $u^N_{m+1}$ be expressed by \eqref{eq:3.4} and \eqref{eq:5.3}, respectively. Suppose that Assumption \ref{as:2.1}-\ref{as:2.2} are satisfied. Suppose Corollary \ref{co:3} and Theorem \ref{th:5} hold. Let $N_1=\left[N^{\frac{\gamma+\alpha}{\gamma}}\right]$. Then
\begin{eqnarray*}
&&\left\|u(t_{m+1})-u^N_{m+1}\right\|_{L^2(D,U)}\\
&&\lesssim\tau\left(T^H+\left\|u_0\right\|_{L^2(D,\dot{U}^{\gamma+\alpha})}+\left\|v_0\right\|_{L^2(D,\dot{U}^{\gamma})}\right)\\
&&~~~+N^{-\frac{\gamma+\alpha}{d}}\left(\frac{T^H}{\varepsilon}+\left\|u_0\right\|_{L^2(D,\dot{U}^{\gamma+\alpha})}+\left\|v_0\right\|_{L^2(D,\dot{U}^{\gamma})}\right), \gamma>\alpha
\end{eqnarray*}
and
\begin{eqnarray*}
&&\left\|u(t_{m+1})-u^N_{m+1}\right\|_{L^2(D,U)}\\
&&\lesssim\tau^{\frac{\gamma}{\alpha}}\left(\frac{T^H}{\varepsilon}+\left\|u_0\right\|_{L^2(D,\dot{U}^{\gamma+\alpha})}+\left\|v_0\right\|_{L^2(D,\dot{U}^{\gamma})}\right)\\
&&~~~+N^{-\frac{\gamma+\alpha}{d}}\left(\frac{T^H}{\varepsilon}+\left\|u_0\right\|_{L^2(D,\dot{U}^{\gamma+\alpha})}+\left\|v_0\right\|_{L^2(D,\dot{U}^{\gamma})}\right), \gamma\le \alpha.
\end{eqnarray*}
\end{proposition}
The proof of this proposition is similar to the one of Theorem \ref{th:7} given in Appendix \ref{sec:C}.

As $\gamma>\alpha$, the derivative of $u(t)$ is time H\"older continuous in the sense of mean-squared $L^p$-norm, which means that the scheme \eqref{eq:5.1} is not optimal to discretize \eqref{eq:3.1} in time. Thus we can design a higher order scheme for the time discretization, as $v(t)$ belongs to $L^p(D,U)$. By modifying the scheme \eqref{eq:5.1}, we can get a better convergence rate than one of \eqref{eq:5.1} in time. The modified scheme is as follows
\begin{eqnarray}\label{eq:5.2-2}
\left[\begin{array}{l}
z^N_{1}\\[1.5mm]
\bar{z}^N_{1}
\end{array} \right]&=&\left[\begin{array}{cc}
\cos\left(A^{\frac{\alpha}{2}}\tau\right)&A^{-\frac{\alpha}{2}}\sin\left(A^{\frac{\alpha}{2}}\tau\right)\\[1.5mm]
-A^{\frac{\alpha}{2}}\sin\left(A^{\frac{\alpha}{2}}\tau\right)&\cos\left(A^{\frac{\alpha}{2}}\tau\right)
\end{array} \right]\left[\begin{array}{c}z^N_{0}\\[1.5mm]
\bar{z}^N_{0}
\end{array} \right]\\
&&+\left[\begin{array}{c}
A^{-\alpha}\left(1-\cos\left(A^{\frac{\alpha}{2}}\tau\right)\right)f_N\left(u^N_0\right)\\[1.5mm]
A^{-\frac{\alpha}{2}}\sin\left(A^{\frac{\alpha}{2}}\tau\right)f_N\left(u^N_0\right)\nonumber
\end{array} \right]
\end{eqnarray}
and for $m\ge1$,
\begin{eqnarray}\label{eq:5.2-02}
\left[\begin{array}{l}
z^N_{m+1}\\[1.5mm]
\bar{z}^N_{m+1}
\end{array} \right]&=&\left[\begin{array}{cc}
\cos\left(A^{\frac{\alpha}{2}}\tau\right)&A^{-\frac{\alpha}{2}}\sin\left(A^{\frac{\alpha}{2}}\tau\right)\\[1.5mm]
-A^{\frac{\alpha}{2}}\sin\left(A^{\frac{\alpha}{2}}\tau\right)&\cos\left(A^{\frac{\alpha}{2}}\tau\right)
\end{array} \right]\left[\begin{array}{c}z^N_{m}\\
\bar{z}^N_{m}
\end{array} \right]\\
&&+\left[\begin{array}{c}
A^{-\alpha}\left(1-\cos\left(A^{\frac{\alpha}{2}}\tau\right)\right)f_N\left(u^N_m\right)\\[1.5mm]
A^{-\frac{\alpha}{2}}\sin\left(A^{\frac{\alpha}{2}}\tau\right)f_N\left(u^N_m\right)
\end{array} \right]\nonumber\\
&&+\left[\begin{array}{c}
\frac{\tau A^{-\alpha}-A^{-\frac{3\alpha}{2}}\sin\left(A^{\frac{\alpha}{2}}\tau\right)}{\tau} \left(f_N\left(u^N_m\right)-f_N\left(u^N_{m-1}\right)\right)\nonumber\\[1.5mm]
\frac{A^{-\alpha}-A^{-\alpha}\cos\left(A^{\frac{\alpha}{2}}\tau\right)}{\tau} \left(f_N\left(u^N_m\right)-f_N\left(u^N_{m-1}\right)\right)
\end{array} \right].\nonumber
\end{eqnarray}
Then a classic application of recursion gives
\begin{eqnarray}\label{eq:5.3}
u^N_{m+1}&=&\cos\left(A^{\frac{\alpha}{2}}t_{m+1}\right)u^N_0+A^{-\frac{\alpha}{2}}\sin\left(A^{\frac{\alpha}{2}}t_{m+1}\right)v^N_0\\
&&+\sum^{m}_{j=0}A^{-\alpha}\left(\cos\left(A^{\frac{\alpha}{2}}(t_{m+1}-t_{j+1})\right)-\cos\left(A^{\frac{\alpha}{2}}(t_{m+1}-t_{j})\right)\right)f_N\left(u^N_j\right)\nonumber\\
&&+\sum^{m}_{j=1} A^{-\alpha}\cos\left(A^{\frac{\alpha}{2}}(t_{m+1}-t_{j+1})\right)\left(f_N\left(u^N_j\right)-f_N\left(u^N_{j-1}\right)\right)\nonumber\\
&&-\sum^{m}_{j=1}\frac{A^{-\alpha}\int^{t_{j+1}}_{t_{j}}\cos\left(A^{\frac{\alpha}{2}}(t_{m+1}-s)\right)\mathrm{d}s}{\tau}\left(f_N\left(u^N_j\right)-f_N\left(u^N_{j-1}\right)\right)\nonumber\\
&&+P_{N_1}\int^{t_{m+1}}_0A^{-\frac{\alpha}{2}}\sin\left(A^{\frac{\alpha}{2}}(t_{m+1}-s)\right)\mathrm{d}B_H(s)\nonumber
\end{eqnarray}
and
\begin{eqnarray}\label{eq:5.3-1}
v^N_{m+1}&=&-A^{\frac{\alpha}{2}}\sin\left(A^{\frac{\alpha}{2}}t_{m+1}\right)u^N_0+\cos\left(A^{\frac{\alpha}{2}}t_{m+1}\right)v^N_0\\
&&-\sum^{m}_{j=0}A^{-\frac{\alpha}{2}}\left(\sin\left(A^{\frac{\alpha}{2}}(t_{m+1}-t_{j+1})\right)-\sin\left(A^{\frac{\alpha}{2}}(t_{m+1}-t_{j})\right)\right)f_N\left(u^N_j\right)\nonumber\\
&&-\sum^{m}_{j=1} A^{-\frac{\alpha}{2}}\sin\left(A^{\frac{\alpha}{2}}(t_{m+1}-t_{j+1})\right)\left(f_N\left(u^N_j\right)-f_N\left(u^N_{j-1}\right)\right)\nonumber\\
&&+\sum^{m}_{j=1}\frac{A^{-\frac{\alpha}{2}}\int^{t_{j+1}}_{t_{j}}\sin\left(A^{\frac{\alpha}{2}}(t_{m+1}-s)\right)\mathrm{d}s}{\tau}\left(f_N\left(u^N_j\right)-f_N\left(u^N_{j-1}\right)\right)\nonumber\\
&&+P_{N_1}\int^{t_{m+1}}_0\cos\left(A^{\frac{\alpha}{2}}(t_{m+1}-s)\right)\mathrm{d}B_H(s).\nonumber
\end{eqnarray}
For $\frac{1}{2}<H<1$ and $\gamma>\alpha$, if the time steps of  \eqref{eq:5.2-11} and \eqref{eq:5.2-02} are the same, the desired convergence rate can not be got. Thus, a more precise approximation of stochastic integral than \eqref{eq:5.2-11} is expected. The inequality $\cos(t_{m+1}-r)-\cos(t_{m+1}-t_{j})\lesssim |r-t_{j}|^\theta \ (0\le\theta\le1)$ and the error equation $\int^{t_{j+1}}_{t_{j}}\int^{s}_{t_{j}}\cos\left(A^{\frac{\alpha}{2}}(t_{m+1}-r)\right)\mathrm{d}r\mathrm{d}B_H(s) $ of \eqref{eq:5.2-11} imply that we can improve the accuracy of approximation for stochastic integral by the following scheme, that is,
\begin{equation}\label{eq:5.9-0}
P_{N_1}\sum^m_{j=0}\int^{t_{j+1}}_{t_{j}}\left(A^{-\frac{\alpha}{2}}\sin\left(A^{\frac{\alpha}{2}}(t_{m+1}-t_j)\right)-(s-t_j)\cos\left(A^{\frac{\alpha}{2}}(t_{m+1}-t_{j})\right)\right)\mathrm{d}B_H(s),
\end{equation}
which ensures the implementation of scheme \eqref{eq:5.2-02} without loss of convergence rate; and Equation \eqref{eq:5.9-0} is easy to simulate by using its explicit variance. Using Lemma \ref{le:02}, the following error estimate of the approximation for stochastic integral is obtained.
\begin{equation}\label{eq:5.9-1}
\begin{split}
\mathrm{E}&\left[\left\|P_{N_1}\sum^m_{j=0}\int^{t_{j+1}}_{t_{j}}\left(A^{-\frac{\alpha}{2}}\sin\left(A^{\frac{\alpha}{2}}(t_{m+1}-s)\right)\right.\right.\right.\\
&-\left.\left.\left.A^{-\frac{\alpha}{2}}\sin\left(A^{\frac{\alpha}{2}}(t_{m+1}-t_j)\right)+(s -t_j)\cos\left(A^{\frac{\alpha}{2}}(t_{m+1}-t_{j})\right)\right)\mathrm{d}B_H(s)\right\|^2\right]\\
=&\mathrm{E}\left[\left\|P_{N_1}\sum^m_{j=0}\int^{t_{j+1}}_{t_j}\int^s_{t_j}\left(\cos\left(A^{\frac{\alpha}{2}}(t_{m+1}-t_{j})\right)\right.\right.\right.\\
&-\left.\left.\left.\cos\left(A^{\frac{\alpha}{2}}(t_{m+1}-r_1)\right)\right)\mathrm{d}r_1\mathrm{d}B_H(s)\right\|^2\right]\\
\lesssim&\sum^{N_1}_{i=0}\lambda_i ^{-2\rho}\sum^m_{j=0}\int^{t_{j+1}}_{t_j}\int^s_{t_j}\left|\cos\left(\lambda_i ^{\frac{\alpha}{2}}(t_{m+1}-t_{j})\right)-\cos\left(\lambda_i ^{\frac{\alpha}{2}}(t_{m+1}-r_1)\right)\right|\mathrm{d}r_1\\
&\times\sum^m_{k=0}\int^{t_{k+1}}_{t_k}\int^t_{t_k}\left|\cos\left(\lambda_i ^{\frac{\alpha}{2}}(t_{m+1}-t_{k})\right)-\cos\left(\lambda_i ^{\frac{\alpha}{2}}(t_{m+1}-r)\right)\right|\mathrm{d}r
|s-t|^{2H-2}\mathrm{d}s\mathrm{d}t\\
\lesssim&\tau^{\min\left\{\frac{2\gamma}{\alpha},4\right\}}\sum^{N_1}_{i=0}\lambda_i ^{\min\{\gamma-\alpha,\alpha\}-2\rho}\sum^m_{j=0}\int^{t_{j+1}}_{t_j}\sum^m_{k=0}\int^{t_{k+1}}_{t_k}|s-t|^{2H-2}\mathrm{d}s\mathrm{d}t\\
=&\tau^{\min\left\{\frac{2\gamma}{\alpha},4\right\}}\sum^{N_1}_{i=0}\lambda_i ^{\min\{\gamma-\alpha,\alpha\}-2\rho}\int^{t_{m+1}}_0\int^{t_{m+1}}_0|s-t|^{2H-2}\mathrm{d}s\mathrm{d}t\\
\lesssim& t^{2H}_{m+1}\tau^{\min\left\{\frac{2\gamma}{\alpha},4\right\}}\sum^{N_1}_{i=0}\lambda_i ^{\min\{\gamma-\alpha,\alpha\}-2\rho},
\end{split}
\end{equation}
the second inequality of which uses the fact
\begin{equation*}
\cos\left(\lambda_i ^{\frac{\alpha}{2}}(t_{m+1}-s)\right)-\cos\left(\lambda_i ^{\frac{\alpha}{2}}(t_{m+1}-t)\right)\lesssim\lambda_i^{\min\{\frac{\gamma-\alpha}{2},\frac{\alpha}{2}\}}|s-t|^{\min\left\{\frac{\gamma-\alpha}{\alpha},1\right\}}.
\end{equation*}

We end this section by showing the error estimates of the fully discrete scheme \eqref{eq:5.2-02} in $L^2(D,U)$ norm.

\begin{theorem} \label{th:7}
Let $u(t_{m+1})$ and $u^N_{m+1}$ be expressed by \eqref{eq:3.4} and \eqref{eq:5.3}, respectively. Suppose that $f(u)\in C^1(\mathbb{R})$ and $f'(u)$ satisfies the Lipschitz condition, and  the conditions of Corollary \ref{co:3} and Theorem \ref{th:5} are satisfied. Take $N>1$ and $0<\tau<1$. If $\gamma>\alpha$ and $N_1=\left[N^{\frac{\gamma+\alpha}{\gamma}}\right]$, then

$\mathrm{(i)}$ for $\alpha<\gamma\le2\alpha$,
\begin{eqnarray*}
&&\left\|u(t_{m+1})-u^N_{m+1}\right\|_{L^2(D,U)}\\
&&\lesssim\tau^{\frac{\alpha+2\rho-\frac{d}{2}}{\alpha}}\left(T^H|\log(\tau)|+\left\|u_0\right\|_{L^2(D,\dot{U}^{\gamma+\alpha})}+\left\|v_0\right\|_{L^2(D,\dot{U}^{\gamma})}\right)\\
&&~~~+N^{-\frac{2\rho+2\alpha-\frac{d}{2}}{d}}\left(T^H\log(N)+\left\|u_0\right\|_{L^2(D,\dot{U}^{\gamma+\alpha})}+\left\|v_0\right\|_{L^2(D,\dot{U}^{\gamma})}\right);\\
\end{eqnarray*}

$\mathrm{(ii)}$ for $\gamma>2\alpha$,
\begin{eqnarray*}
&&\left\|u(t_{m+1})-u^N_{m+1}\right\|_{L^2(D,U)}\\
&&\lesssim\tau^{2}\left(T^H+\left\|u_0\right\|_{L^2(D,\dot{U}^{\gamma+\alpha})}+\left\|v_0\right\|_{L^2(D,\dot{U}^{\gamma})}\right)\\
&&~~~+N^{-\frac{2\rho+2\alpha-\frac{d}{2}}{d}}\left(T^H\log(N)+\left\|u_0\right\|_{L^2(D,\dot{U}^{\gamma+\alpha})}+\left\|v_0\right\|_{L^2(D,\dot{U}^{\gamma})}\right).\\
\end{eqnarray*}
\end{theorem}
The detailed proof of Theorem \ref{th:7} is given in Appendix \ref{sec:C}.

Proposition \ref{Pro:6} and Theorem \ref{th:7} show that one can choose the appropriate technique to solve \eqref{eq:1.05}, i.e., when $\alpha\le\gamma$, use \eqref{eq:5.1} to discretize \eqref{eq:1.05}, and if $\alpha>\gamma$, the scheme \eqref{eq:5.2-02} can be chosen to obtain the approximation of $u(t)$. In fact, when $u_0\in{L^2(D,\dot{U}^{\gamma})}$ and $v_0\in{L^2(D,\dot{U}^{\gamma-\alpha})}$, the temporal rates of convergence still hold in Proposition \ref{Pro:6} and Theorem \ref{th:7}.

\section{Numerical experiments}\label{sec:6}

In this section, we present numerical examples to verify the theoretical results and the effect of the parameters $\alpha$ and $\rho$ on the convergence. All numerical errors are given in the sense of mean-squared $L^2$-norm.

We solve \eqref{eq:1.05} in the two-dimensional domain $D=(0,1)\times(0,1)$ by the proposed
scheme \eqref{eq:5.2-02} with $x=(x_1,x_2)$, the smooth initial data $u_0=\frac{\sin(\pi x_1)\sin(\pi x_2)}{2}$, and $v_0=\sin(4\pi x_1)\sin(4\pi x_2)$. In $D=(0,1)\times(0,1)$, the Dirichlet eigenpairs of $-\mathrm{\Delta}$ are $\lambda_{i,j}=\pi^2(i^2+j^2)$ and $\phi_{i,j}=2\sin(i\pi x_1)\sin(j\pi x_2)$ with $i,j=1,2,\dots,N$. Unless otherwise specified, we choose $f(u(t))=u(t)$. To calculate the convergence orders, the following formulas are used.
\begin{equation*}
\begin{split}
 \textrm{convergence rate in space}&=
 \frac{\ln\left(\left\|u^{aN}_{M}-u^{N}_{M}\right\|_{L^2(D,U)}/
\left\|u^{N}_{M}-u^{N/a}_{M}\right\|_{L^2(D,U)}\right)}{\ln a},\\
  \textrm{convergence rate in time}&=
 \frac{\ln\left(\left\|u^{N}_{aM}-u^{N}_{M}\right\|_{L^2(D,U)}/
\left\|u^{N}_{M}-u^{N}_{M/a}\right\|_{L^2(D,U)}\right)}{\ln a},
\end{split}
\end{equation*}
where the constant $a>1$.
In numerical simulations, the errors $\left\|u^{aN}_{M}-u^{N}_{M}\right\|_{L^2(D,U)}$ are calculated by Monte Carlo method, i.e.,
\begin{equation*}
\begin{split}
\left\|u^{aN,M}_{M}-u^{N,M}_{M}\right\|_{L^2(D,U)}&=\left(\mathrm{E}\left[\left\|u^{aN}_{M}-u^{N}_{M}\right\|^2\right]\right)^{\frac{1}{2}}\\
&\approx \left(\frac{1}{K}\sum^K_{k=1}\left\|u^{aN}_{M,k}-u^{N}_{M,k}\right\|^2\right)^{\frac{1}{2}}.
\end{split}
\end{equation*}
We take $K=1000$ as the number of the simulation trajectories. The symbol $k$ represents the $k$-th trajectory.

\begin{table}[H]
 \renewcommand\arraystretch{1.6}
\caption{Spatial convergence rates with $T=0.3$, $M =900$, $H=0.5$ and $\rho=1$.}\label{table:1}
\centering
\begin{tabular}{c c c c c c c c c}
\Xhline{1.2pt}
&$N$ & $\alpha=0.4$ & Rate &$\alpha=0.6$ &Rate & $\alpha=0.8$ &Rate  \\
\hline
$N_1=N$ &256&2.610e-04&     & 1.064e-04&     &4.810e-05&          \\
  &576&1.608e-04& 0.597& 5.910e-05& 0.725&2.365e-05& 0.875\\
  &1296&9.545e-05& 0.643& 3.220e-05& 0.749&1.173e-05& 0.865\\
\hline
$N_1=\left[N^{\frac{\gamma+\alpha}{\gamma}}\right]$ &256&1.051e-04&     & 2.367e-05&     &5.912e-06&          \\
  &576&4.787e-05& 0.970& 9.867e-06& 1.079&2.073e-06& 1.292\\
  &1296&2.175e-05& 0.973& 4.055e-06& 1.097&7.343e-07& 1.280\\
\hline
\end{tabular}
\end{table}
The spatial convergence rates of the scheme \eqref{eq:5.2-02} is tested with the end time $T=0.3$ and $M=900$, which ensures the spatial error is the dominant one.
In Table \ref{table:1}, one can see that the spatial convergence rates tend to $\frac{2\rho+\alpha-1}{2}$ if $N_1=N$, and the convergence rates are approximately equal to $\frac{2\rho+2\alpha-1}{2}$ after postprocessing the stochastic integral $\left(N_1=\left[N^{\frac{\gamma+\alpha}{\gamma}}\right]\right)$. And the convergence rates of the spectral Galerkin method are improved, as $\alpha$ increases. The numerical results verify the theoretical ones.

\begin{figure}[htbp]
  \centering
  \label{fig:a}\includegraphics[scale=0.55]{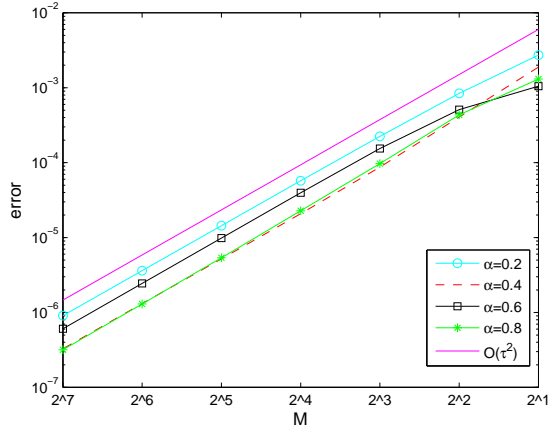}
  \caption{Temporal error convergence of the modified stochastic trigonometric method for the space-time white noise $(H=0.5)$.}
  \label{fig:fig1}
\end{figure}
Next, we observe the behavior of the temporal convergence. We solve the problem \eqref{eq:1.05} by using the scheme \eqref{eq:5.2-02} with $f(u(t))=u^2(t)$, $\rho=2.5$, $T=0.6$, and $N=400$ in Figures \ref{fig:fig1} and \ref{fig:fig2}.
The sufficiently big $\rho$ and $N$ guarantee that the dominant errors arise from the temporal approximation. As $H=\frac{1}{2}$, the simulation of the stochastic integral $\int^T_0\sin\left(\lambda_{i,j} (T-t) \right)\mathrm{d}\beta(t)$ is easily implementable by using explicit variance of the stochastic integral, which is $\frac{T}{2}-\frac{\sin\left(2\lambda_{i,j}T\right)}{4\lambda_{i,j}}$.
For $H\in\left(\frac{1}{2},1\right)$, one can obtain the approximation of the stochastic integral by using scheme \eqref{eq:5.9-0}. The simulation of the approximation is given in Appendix \ref{sec:B}. Figures \ref{fig:fig1} and \ref{fig:fig2} show that the temporal convergence rates have an order of $2$ by using the proposed scheme, as $\gamma>2\alpha$, and the convergence rates are independent of $H$.

\begin{figure}[htbp]
  \centering
  \label{fig:b}\includegraphics[scale=0.55]{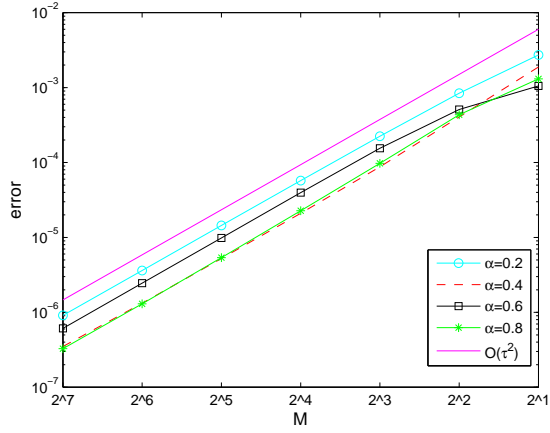}
  \caption{Temporal error convergence of the modified stochastic trigonometric method for the space-time fractional Gaussian noise $(H=0.75)$.}
  \label{fig:fig2}
\end{figure}

\begin{table}[H]
\renewcommand\arraystretch{1.6}
\caption{Time convergence rates with $N=10^6$, $T=0.6$, $H=0.5$, and $\rho=0.68$.}\label{table:2}
\centering
\begin{tabular}{c c c c c c c c }
\Xhline{1.2pt}
$M$ & $\alpha=0.5$ & Rate &$\alpha=0.7$ &Rate & $\alpha=0.9$ &Rate  \\
\hline
  4&1.400e-03&     & 2.414e-03&     &2.980e-03&          \\
  8&4.158e-04& 1.751& 8.575e-04& 1.493&1.134e-03&1.394\\
  16&1.220e-04& 1.769& 3.077e-04& 1.479&4.341e-04&1.385  \\
  \hline
\end{tabular}
\end{table}
As $\alpha<\gamma\le2\alpha$, the convergence rates of the proposed scheme are close to $\frac{\alpha+2\rho-1}{\alpha}$ in time.  For $H=\frac{1}{2}$, from Table \ref{table:2}, one can see that the temporal convergence rates reduce with the increase of $\alpha$, for fixing $\rho$. As $\alpha=0.5$, $0.7$, and $0.9$, the theoretical convergence rates are approximately  $1.720$, $1.514$, and $1.400$, respectively. For $H=0.6$, Table \ref{table:3} demonstrates that the time convergence rates increase with the increase of $\rho$, for fixing $\alpha$. The temporal theoretical convergence rates are approximately $1.556$, $1.667$, and $1.778$, for $\rho=0.75$, $0.80$, and $0.85$, respectively. Tables \ref{table:2} and \ref{table:3} show that  the numerical results confirm the error estimates in Theorem \ref{th:7}.

\begin{table}[H]
\renewcommand\arraystretch{1.6}
\caption{Time convergence rates with $N=10^6$, $T=0.6$, $H=0.6$, and $\alpha=0.9$.}\label{table:3}
\centering
\begin{tabular}{c c c c c c c c }
\Xhline{1.2pt}
$M$ & $\rho=0.75$ & Rate &$\rho=0.8$ &Rate & $\rho=0.85$ &Rate  \\
\hline
  4&4.568e-03 &     & 3.543e-03&     &2.907e-03&          \\
  8&1.559e-03 & 1.551&1.156e-03 & 1.616 &8.956e-04&1.699  \\
  16&5.330e-04 & 1.548&3.733e-04&1.630 &2.704e-04&1.728  \\
  \hline
\end{tabular}
\end{table}

\section{Conclusion}\label{sec:7}

This paper discusses the numerical schemes and their error analyses for the equation describing the wave propagation with attenuation and possible external disturbance. Two kinds of external noises (white noise and fractional Gaussian noise) are considered. The regularity results of the mild solution of the equation are obtained.  The spectral Galerkin method is used for space approximation and the stochastic trigonometric method for time approximation. The detailed error analyses are performed. The techniques of equivalent transformation and postprocessing the stochastic integral improve the convergence rate in space from  $(2\rho+\alpha-\frac{d}{2})/d$ to $(2\rho+2\alpha-\frac{d}{2})/d$.  For the temporal approximation, by modifying the stochastic trigonometric method, when $\gamma>\alpha$, the superlinear convergence is obtained. The convergence rates of the designed schemes are independent of hurst index $H$. The extensive numerical experiments confirm the theoretical results.


\appendix

\section{Proof of Theorem 5.2}\label{sec:C}
\begin{proof}
First, combining \eqref{eq:4.5}, \eqref{eq:5.3}, \eqref{eq:5.9-1}, the assumptions, and Corollary \ref{co:3} leads to
\begin{eqnarray}\label{eq:5.9-2}
~~~&&\
\left\|u^N(t_{m+1})-u^N_{m+1}\right\|_{L^2(D,U)}\\
&&\lesssim\left\|\sum^{m}_{j=1}\int^{t_{j+1}}_{t_j}A^{-\frac{\alpha}{2}}\sin\left(A^{\frac{\alpha}{2}}(t_{m+1}-s)\right)\left[f_N\left(u^N(s)\right)\right.\right.\nonumber\\
&&~~~-\left.\left.f_N\left(u^N(t_j)\right)-\int^s_{t_j}f'_N\left(u^N(t_j)\right)v^N(t_j)\mathrm{d}r\right]\mathrm{d}s\right\|_{L^2(D,U)}\nonumber\\
&&~~~+\left\|\sum^{m}_{j=1}A^{-\alpha}\frac{ \tau\cos\left(A^{\frac{\alpha}{2}}(t_{m+1}-t_{j+1})\right)-\int^{t_{j+1}}_{t_j}\cos\left(A^{\frac{\alpha}{2}}(t_{m+1}-s)\right)\mathrm{d}s}{\tau}\right.\nonumber\\
&&~~~\times\left.\left[\tau f'_N\left(u^N(t_j)\right)v^N(t_j)-f_N\left(u^N_j\right)+f_N\left(u^N_{j-1}\right)\right]\right\|_{L^2(D,U)}\nonumber\\
&&~~~+\left\|\sum^{m}_{j=1}\int^{t_{j+1}}_{t_j}A^{-\frac{\alpha}{2}}\sin\left(A^{\frac{\alpha}{2}}(t_{m+1}-s)\right)\left[f_N\left(u^N(t_j)\right)-f_N\left(u^N_j\right)\right]\mathrm{d}s\right\|_{L^2(D,U)}\nonumber\\
&&~~~+\left\|\int^{t_{1}}_{0}A^{-\frac{\alpha}{2}}\sin\left(A^{\frac{\alpha}{2}}(t_{m+1}-s)\right)\left[f_N\left(u^N(s)\right)-f_N\left(u^N_0\right)\right]\mathrm{d}s\right\|_{L^2(D,U)}\nonumber\\
&&~~~+\tau^{\min\left\{\frac{\gamma}{\alpha},2\right\}}\left(\sum^{N_1}_{i=0}\lambda_i ^{\min\{\gamma-\alpha,\alpha\}-2\rho}\right)^{\frac{1}{2}}\nonumber\\
&&\lesssim J_1+J_2+\sum^{m}_{j=0}\tau\left\|u^N(t_j)-u^N_j\right\|_{L^2(D,U)}+\tau^2\nonumber\\
&&~~~+\tau^{\min\left\{\frac{\gamma}{\alpha},2\right\}}\left(\sum^{N_1}_{i=0}\lambda_i ^{\min\{\gamma-\alpha,\alpha\}-2\rho}\right)^{\frac{1}{2}}.\nonumber
\end{eqnarray}

When $\frac{1}{2}<H<1$, let $\theta=\min\left\{\frac{\gamma-\alpha}{\alpha},1\right\}$. We have
\begin{eqnarray*}
J_1&\lesssim&\left\|\sum^{m}_{j=1}\int^{t_{j+1}}_{t_j}A^{-\frac{\alpha}{2}}\sin\left(A^{\frac{\alpha}{2}}(t_{m+1}-s)\right)\right.\nonumber\\
&&\times\left.\left[\int^s_{t_j}f'_N\left(u^N(r)\right)v^N(r)\mathrm{d}r-\int^s_{t_j}f'_N\left(u^N(t_j)\right)v^N(t_j)\mathrm{d}r\right]\mathrm{d}s\right\|_{L^2(D,U)}\nonumber\\
&\lesssim&\left\|\sum^{m}_{j=1}\int^{t_{j+1}}_{t_j}A^{-\frac{\alpha}{2}}\sin\left(A^{\frac{\alpha}{2}}(t_{m+1}-s)\right)\right.\\
&&\times\left.\int^s_{t_j}\left(f'_N\left(u^N(r)\right)-f'_N\left(u^N(t_j)\right)\right)v^N(r)\mathrm{d}r\mathrm{d}s\right\|_{L^2(D,U)}\nonumber\\
&&+\left\|\sum^{m}_{j=1}\int^{t_{j+1}}_{t_j}A^{-\frac{\alpha}{2}}\sin\left(A^{\frac{\alpha}{2}}(t_{m+1}-s)\right)\right.\nonumber\\
&&\times\left.\int^s_{t_j}f'_N\left(u^N(t_j)\right)\left(\cos\left(A^{\frac{\alpha}{2}}(r-t_j)\right)-I\right)v^N(t_j)\mathrm{d}r\mathrm{d}s\right\|_{L^2(D,U)}\nonumber\\
&&+\left\|\sum^{m}_{j=1}\int^{t_{j+1}}_{t_j}A^{-\frac{\alpha}{2}}\sin\left(A^{\frac{\alpha}{2}}(t_{m+1}-s)\right)\right.\nonumber\\
&&\times\left.\int^s_{t_j}f'_N\left(u^N(t_j)\right)\left(v^N(r)-\cos\left(A^{\frac{\alpha}{2}}(r-t_j)\right)v^N(t_j)\right)\mathrm{d}r\mathrm{d}s\right\|_{L^2(D,U)}\nonumber\\
&\lesssim&\sum^{m}_{j=1}\int^{t_{j+1}}_{t_j}\int^s_{t_j}\left\|\left(u^N(r)-u^N(t_j)\right)\times v^N(r)\right\|_{L^2(D,U)}\mathrm{d}r\mathrm{d}s\nonumber\\
&&+\sum^{m}_{j=1}\tau^\theta\int^{t_{j+1}}_{t_j}\int^s_{t_j}\left\|A^{\frac{\theta\alpha}{2}}v^N(t_j)\right\|_{L^2(D,U)}\mathrm{d}r\mathrm{d}s+II\nonumber\\
&\lesssim&\sum^{m}_{j=1}\int^{t_{j+1}}_{t_j}\int^s_{t_j}\int^r_{t_j}\left\|v^N(t)v^N(r)\right\|_{L^2(D,U)}\mathrm{d}t\mathrm{d}r\mathrm{d}s\\
&&+\sum^{m}_{j=1}\tau^{2+\theta}\left\|A^{\frac{\theta\alpha}{2}}v^N(t_j)\right\|_{L^2(D,U)}+II\nonumber\\
&\lesssim&\sum^{m}_{j=1}\int^{t_{j+1}}_{t_j}\int^s_{t_j}\int^r_{t_j}\left\|\left|v^N(t)\right|^2+\left|v^N(r)\right|^2\right\|_{L^2(D,U)}\mathrm{d}t\mathrm{d}r\mathrm{d}s\\
&&+\sum^{m}_{j=1}\tau^{2+\theta}\left\|A^{\frac{\theta\alpha}{2}}v^N(t_j)\right\|_{L^2(D,U)}+II\nonumber\\
&\lesssim&\tau^2\left(1+\left\|u_0\right\|^2_{L^4(D,\dot{U}^\gamma)}+\left\|v_0\right\|^2_{L^4(D,\dot{U}^{\gamma-\alpha})}\right)+\sum^{m}_{j=1}\tau^{2+\theta}\left\|A^{\frac{\theta\alpha}{2}}v^N(t_j)\right\|_{L^2(D,U)}+II.\\
\end{eqnarray*}

The condition $\gamma>\alpha$ implies that $\rho>\frac{d}{4}$. Then combining the fact that $\{\beta^{i}_{H}(t)\}_{i\in\mathbb{N}}$ are mutually independent, Lemma \ref{le:02}, and \eqref{eq:4.7-1}, we have
\begin{eqnarray*}
II&=&\left\|\sum^{m}_{j=1}\int^{t_{j+1}}_{t_j}A^{-\frac{\alpha}{2}}\sin\left(A^{\frac{\alpha}{2}}(t_{m+1}-s)\right)\int^s_{t_j}f'_N\left(u^N(t_j)\right)\right.\\
&&\times\left.\left[-A^{\frac{\alpha}{2}}\sin\left(A^{\frac{\alpha}{2}}(r-t_j)\right)u^N(t_j)+\int^{r}_{t_j}\cos\left(A^{\frac{\alpha}{2}}(r-y)\right)f_N\left(u^N(y)\right)\mathrm{d}y\right.\right.\\
&&+\left.\left.\int^{r}_{t_j}\sum^{N_1}_{i=1}\cos\left(\lambda_i^{\frac{\alpha}{2}}(r-y)\right)\sigma_i\phi_i(x)\mathrm{d}\beta_H^i(y)\right]\mathrm{d}r\mathrm{d}s\right\|_{L^2(D,U)}\\
&\lesssim&\tau^{2+\theta}\sum^{m}_{j=1}\left\|A^{\frac{\alpha(\theta+1)}{2}}u^N(t_j)\right\|_{L^2(D,U)}+\tau^2\left(T^H+\left\|u_0\right\|_{L^2(D,\dot{U}^\gamma)}+\left\|v_0\right\|_{L^2(D,\dot{U}^{\gamma-\alpha})}\right)\\
&&+\left(\sum^{N_1}_{i=1}\lambda_i^{-2\rho}\mathrm{E}\left[\left\|\sum^{m}_{j=1}\int^{t_{j+1}}_{t_j}A^{-\frac{\alpha}{2}}\sin\left(A^{\frac{\alpha}{2}}(t_{m+1}-s)\right)\int^s_{t_j}f'_N\left(u^N(t_j)\right)\right.\right.\right.\\
&&\times\left.\left.\left.\int^{r}_{t_j}\cos\left(\lambda_i^{\frac{\alpha}{2}}(r-y)\right)\phi_i(x)\mathrm{d}\beta_H^i(y)\mathrm{d}r\mathrm{d}s\right\|^2\right]\right)^{\frac{1}{2}}\\
&\lesssim&\left(\sum^{N_1}_{i=1}\lambda_i^{-2\rho}\mathrm{E}\left[\int_D\sum^{m}_{j=1}\sum^{m}_{k=1}\int^{t_{j+1}}_{t_j}\int^s_{t_j}\int^{t_{k+1}}_{k_j}\int^t_{k_j}\int^{r}_{t_j}\int^{r_1}_{t_k}A^{-\frac{\alpha}{2}}\sin\left(A^{\frac{\alpha}{2}}(t_{m+1}-s)\right)\right.\right.\\
&&\times\left.\left.f'_N\left(u^N(t_j)\right)\cos\left(\lambda_i^{\frac{\alpha}{2}}(r-y)\right)\phi_i(x)A^{-\frac{\alpha}{2}}\sin\left(A^{\frac{\alpha}{2}}(t_{m+1}-t)\right)f'_N\left(u^N(t_k)\right)\right.\right.\\
&&\times\left.\left.\cos\left(\lambda_i^{\frac{\alpha}{2}}(r_1-y_1)\right)\phi_i(x)|y-y_1|^{2H-2}\mathrm{d}y\mathrm{d}y_1\mathrm{d}r\mathrm{d}s\mathrm{d}r_1\mathrm{d}t\mathrm{d}x\right]\right)^{\frac{1}{2}}\\
&&+\tau^{2+\theta}\sum^{m}_{j=1}\left\|A^{\frac{\alpha(\theta+1)}{2}}u^N(t_j)\right\|_{L^2(D,U)}\\
&\lesssim&\left(\tau^4\sum^{m}_{j=1}\sum^{m}_{k=1}\mathrm{E}\left[\left\|f'_N\left(u^N(t_j)\right)\phi_i(x)\right\|\left\|f'_N\left(u^N(t_k)\right)\phi_i(x)\right\|\right]\right.\\
&&\times\left.\int^{t_{j+1}}_{t_j}\int^{t_{k+1}}_{t_k}|y-y_1|^{2H-2}\mathrm{d}y\mathrm{d}y_1\right)^{\frac{1}{2}}+\tau^{2+\theta}\sum^{m}_{j=1}\left\|A^{\frac{\alpha(\theta+1)}{2}}u^N(t_j)\right\|_{L^2(D,U)}.
\end{eqnarray*}
Combining the above estimates and Corollary \ref{co:3} leads to
\begin{eqnarray*}
J_1&\lesssim & \tau^{\frac{\gamma}{\alpha}}\left(\frac{T^H}{\varepsilon}+\left\|u_0\right\|_{L^2(D,\dot{U}^\gamma)}+\left\|v_0\right\|_{L^2(D,\dot{U}^{\gamma-\alpha})}\right.\\
&&+\left.\left\|u_0\right\|^2_{L^4(D,\dot{U}^\gamma)}+\left\|v_0\right\|^2_{L^4(D,\dot{U}^{\gamma-\alpha})}\right),
\ \alpha<\gamma\le2\alpha
\end{eqnarray*}
and
\begin{eqnarray*}
J_1&\lesssim&\tau^{2}\left(T^H+\left\|u_0\right\|_{L^2(D,\dot{U}^\gamma)}+\left\|v_0\right\|_{L^2(D,\dot{U}^{\gamma-\alpha})}\right.\\
&&+\left.\left\|u_0\right\|^2_{L^4(D,\dot{U}^\gamma)}+\left\|v_0\right\|^2_{L^4(D,\dot{U}^{\gamma-\alpha})}\right),\ \gamma>2\alpha.
\end{eqnarray*}
Similar to $J_1$, one gets
\begin{eqnarray*}
J_2&\lesssim&\left\|\sum^{m}_{j=1}\frac{ A^{-\frac{\alpha}{2}}\int^{t_{j+1}}_{t_j}\int^{t_{j+1}}_{s}\sin\left(A^{\frac{\alpha}{2}}(t_{m+1}-r)\right)\mathrm{d}r\mathrm{d}s}{\tau}\left[f_N\left(u^N(t_{j-1})\right)\right.\right.\\
&&-\left.\left.f_N\left(u^N(t_j)\right)+\tau f'_N\left(u^N(t_j)\right)v^N(t_j)\right]\right\|_{L^2(D,U)}\\
&&+\tau\sum^{m}_{j=1}\left\|\left[f_N\left(u^N(t_j)\right)-f_N\left(u^N(t_{j-1})\right)-f_N\left(u^N_j\right)+f_N\left(u^N_{j-1}\right)\right]\right\|_{L^2(D,U)}\\
&\lesssim&\left\|\sum^{m}_{j=1}\frac{ A^{-\frac{\alpha}{2}}\int^{t_{j+1}}_{t_j}\int^{t_{j+1}}_{s}\sin\left(A^{\frac{\alpha}{2}}(t_{m+1}-r)\right)\mathrm{d}r\mathrm{d}s}{\tau}\right.\\
&&\times\left.\left[\int^{t_j}_{t_{j-1}}f'_N\left(u^N(t_j)\right)v^N(t_j)\mathrm{d}r-\int^{t_j}_{t_{j-1}}f'_N\left(u^N(r)\right)v^N(t_j)\mathrm{d}r\right]\right\|_{L^2(D,U)}\\
&&+\left\|\sum^{m}_{j=1}\frac{ A^{-\frac{\alpha}{2}}\int^{t_{j+1}}_{t_j}\int^{t_{j+1}}_{s}\sin\left(A^{\frac{\alpha}{2}}(t_{m+1}-r)\right)\mathrm{d}r\mathrm{d}s}{\tau}\right.\\
&&\times\left.\left[\int^{t_j}_{t_{j-1}}f'_N\left(u^N(r)\right)\left(v^N(t_j)-\cos\left(A^{\frac{\alpha}{2}}(t_{j}-r)\right) v^N(r)\right)\mathrm{d}r\right]\right\|_{L^2(D,U)}\\
&&+\left\|\sum^{m}_{j=1}\frac{ A^{-\frac{\alpha}{2}}\int^{t_{j+1}}_{t_j}\int^{t_{j+1}}_{s}\sin\left(A^{\frac{\alpha}{2}}(t_{m+1}-r)\right)\mathrm{d}r\mathrm{d}s}{\tau}\right.\\
&&\times\left.\left[\int^{t_j}_{t_{j-1}}f'_N\left(u^N(r)\right)\left(\cos\left(A^{\frac{\alpha}{2}}(t_{j}-r)\right)-I\right)v^N(r)\mathrm{d}r\right]\right\|_{L^2(D,U)}\\
&&+\tau\sum^{m}_{j=0}\left\|u^N(t_j)-u^N_j\right\|_{L^2(D,U)}.
\end{eqnarray*}
For $\alpha<\gamma\le2\alpha$, we have
\begin{eqnarray*}
J_2&\lesssim&\tau^{\frac{\gamma}{\alpha}}\left(\frac{T^H}{\varepsilon}+\left\|u_0\right\|_{L^2(D,\dot{U}^\gamma)}+\left\|v_0\right\|_{L^2(D,\dot{U}^{\gamma-\alpha})}\right.\\
&+&\left.\left\|u_0\right\|^2_{L^4(D,\dot{U}^\gamma)}+\left\|v_0\right\|^2_{L^4(D,\dot{U}^{\gamma-\alpha})}\right)+\tau\sum^{m}_{j=0}\left\|u^N(t_j)-u^N_j\right\|_{L^2(D,U)}.
\end{eqnarray*}
When $\gamma>2\alpha$, we also have
\begin{eqnarray*}
J_2&\lesssim&\tau^{2}\left(T^H+\left\|u_0\right\|_{L^2(D,\dot{U}^\gamma)}+\left\|v_0\right\|_{L^2(D,\dot{U}^{\gamma-\alpha})}\right.\\
&+&\left.\left\|u_0\right\|^2_{L^4(D,\dot{U}^\gamma)}+\left\|v_0\right\|^2_{L^4(D,\dot{U}^{\gamma-\alpha})}\right)+\tau\sum^{m}_{j=0}\left\|u^N(t_j)-u^N_j\right\|_{L^2(D,U)}. \end{eqnarray*}
Combining \eqref{eq:5.9-2}, $J_1$, and $J_2$ leads to
\begin{eqnarray*}\label{eq:5.4}
&&\left\|u^N(t_{m+1})-u^N_{m+1}\right\|_{L^2(D,U)}\\
&&\lesssim\tau^{\frac{\gamma}{\alpha}}\left(\frac{T^H}{\varepsilon}+\left\|u_0\right\|_{L^2(D,\dot{U}^\gamma)}+\left\|v_0\right\|_{L^2(D,\dot{U}^{\gamma-\alpha})}\right.\\
&&~~~+\left.\left\|u_0\right\|^2_{L^4(D,\dot{U}^\gamma)}+\left\|v_0\right\|^2_{L^4(D,\dot{U}^{\gamma-\alpha})}\right),\alpha<\gamma\le2\alpha
\end{eqnarray*}
and
\begin{eqnarray*}\label{eq:5.5}
&&\left\|u^N(t_{m+1})-u^N_{m+1}\right\|_{L^2(D,U)}\\
&&\lesssim\tau^{2}\left(T^H+\left\|u_0\right\|_{L^2(D,\dot{U}^\gamma)}+\left\|v_0\right\|_{L^2(D,\dot{U}^{\gamma-\alpha})}\right.\\
&&~~~+\left.\left\|u_0\right\|^2_{L^4(D,\dot{U}^\gamma)}+\left\|v_0\right\|^2_{L^4(D,\dot{U}^{\gamma-\alpha})}\right), \gamma>2\alpha.
\end{eqnarray*}

When $H=\frac{1}{2}$, using the same steps in \eqref{eq:5.9-2}, we get
\begin{eqnarray*}
&&\left\|u^N(t_{m+1})-u^N_{m+1}\right\|_{L^2(D,U)}\\
&&\lesssim I_1+I_2+\sum^{m}_{j=0}\tau\left\|u^N(t_j)-u^N_j\right\|_{L^2(D,U)}+\tau^2.
\end{eqnarray*}
For $\alpha<\gamma\le2\alpha$, using Corollary \ref{co:3} and the assumptions of $f$, we obtain
\begin{eqnarray}\label{eq:C2}
I_1&\lesssim&\left\|\sum^{m}_{j=1}\int^{t_{j+1}}_{t_j}A^{-\frac{\alpha}{2}}\sin\left(A^{\frac{\alpha}{2}}(t_{m+1}-s)\right)\right.\\
&&\times\left.\int^s_{t_j}\left(f'_N\left(u^N(r)\right)-f'_N\left(u^N(t_j)\right)\right)v^N(r)\mathrm{d}r\mathrm{d}s\right\|_{L^2(D,U)}\nonumber\\
&&+\left\|\sum^{m}_{j=1}\int^{t_{j+1}}_{t_j}A^{-\frac{\alpha}{2}}\sin\left(A^{\frac{\alpha}{2}}(t_{m+1}-s)\right)\right.\nonumber\\
&&\times\left.\int^s_{t_j}f'_N\left(u^N(t_j)\right)\left(\cos\left(A^{\frac{\alpha}{2}}(r-t_j)\right)-I\right)v^N(t_j)\mathrm{d}r\mathrm{d}s\right\|_{L^2(D,U)}\nonumber\\
&&+\left\|\sum^{m}_{j=1}\int^{t_{j+1}}_{t_j}A^{-\frac{\alpha}{2}}\sin\left(A^{\frac{\alpha}{2}}(t_{m+1}-s)\right)\right.\nonumber\\
&&\times\left.\int^s_{t_j}f'_N\left(u^N(t_j)\right)\left(v^N(r)-\cos\left(A^{\frac{\alpha}{2}}(r-t_j)\right)v^N(t_j)\right)\mathrm{d}r\mathrm{d}s\right\|_{L^2(D,U)}\nonumber\\
&\lesssim&\sum^{m}_{j=1}\int^{t_{j+1}}_{t_j}\int^s_{t_j}\left\|\left(u^N(r)-u^N(t_j)\right)\times v^N(r)\right\|_{L^2(D,U)}\mathrm{d}r\mathrm{d}s\nonumber\\
&&+\sum^{m}_{j=1}\tau^{\frac{\gamma-\alpha}{\alpha}}\int^{t_{j+1}}_{t_j}\int^s_{t_j}\left\|A^{\frac{\gamma-\alpha}{2}}v^N(t_j)\right\|_{L^2(D,U)}\mathrm{d}r\mathrm{d}s+III\nonumber\\
&\lesssim&\tau^{\frac{\gamma}{\alpha}}\left(\frac{T^H}{\varepsilon}+\left\|u_0\right\|_{L^2(D,\dot{U}^\gamma)}+\left\|v_0\right\|_{L^2(D,\dot{U}^{\gamma-\alpha})}\right.\nonumber\\
&&+\left.\left\|u_0\right\|^2_{L^4(D,\dot{U}^\gamma)}+\left\|v_0\right\|^2_{L^4(D,\dot{U}^{\gamma-\alpha})}\right)+III.\nonumber
\end{eqnarray}
Combining the fact that $\{\beta^{i}_{H}(t)\}_{i\in\mathbb{N}}$ are mutually independent and Equation \eqref{eq:4.7-1}, we have
\begin{eqnarray*}
III&=&\left\|\sum^{m}_{j=1}\int^{t_{j+1}}_{t_j}A^{-\frac{\alpha}{2}}\sin\left(A^{\frac{\alpha}{2}}(t_{m+1}-s)\right)\int^s_{t_j}f'_N\left(u^N(t_j)\right)\right.\\
&&\times\left.\left[-A^{\frac{\alpha}{2}}\sin\left(A^{\frac{\alpha}{2}}(r-t_j)\right)u^N(t_j)+\int^{r}_{t_j}\cos\left(A^{\frac{\alpha}{2}}(r-y)\right)f_N\left(u^N(y)\right)\mathrm{d}y\right.\right.\\
&&+\left.\left.\int^{r}_{t_j}\sum^{N_1}_{i=1}\cos\left(\lambda_i^{\frac{\alpha}{2}}(r-y)\right)\sigma_i\phi_i(x)\mathrm{d}\beta^i(y)\right]\mathrm{d}r\mathrm{d}s\right\|_{L^2(D,U)}\\
&\lesssim&\tau^{1+\frac{\gamma}{\alpha}}\sum^{m}_{j=1}\left\|A^{\frac{\gamma}{2}}u^N(t_j)\right\|_{L^2(D,U)}+\tau^2\left(T^H+\left\|u_0\right\|_{L^2(D,\dot{U}^\gamma)}+\left\|v_0\right\|_{L^2(D,\dot{U}^{\gamma-\alpha})}\right)\\
&&+\left(\sum^{m}_{j=1}\mathrm{E}\left[\left\|\int^{t_{j+1}}_{t_j}\sum^{N_1}_{i=1}\lambda_i^{-\frac{\alpha}{2}}\sin\left(\lambda_i^{\frac{\alpha}{2}}(t_{m+1}-s)\right)\int^s_{t_j}f'_N\left(u^N(t_j)\right)\right.\right.\right.\\
&&\times\left.\left.\left.\int^{r}_{t_j}\cos\left(\lambda_i^{\frac{\alpha}{2}}(r-y)\right)\sigma_i\phi_i(x)\mathrm{d}\beta^i(y)\mathrm{d}r\mathrm{d}s\right\|^2\right]\right)^{\frac{1}{2}}\\
&\lesssim&\tau^{\frac{\gamma}{\alpha}}\left(\frac{T^H}{\varepsilon}+\left\|u_0\right\|_{L^2(D,\dot{U}^\gamma)}+\left\|v_0\right\|_{L^2(D,\dot{U}^{\gamma-\alpha})}\right).
\end{eqnarray*}
In first inequality, we employ the fact that Brownian motion is a process with independent increment, that is
\begin{eqnarray*}
&\mathrm{E}&\left[\int^{t_{j+1}}_{t_j}\int^s_{t_j}f'_N\left(u^N(t_j)\right)\int^{r}_{t_j}\cos\left(\lambda_i^{\frac{\alpha}{2}}(r-y)\right)\mathrm{d}\beta^i(y)\mathrm{d}r\mathrm{d}s\right.\\
&& \times \left.\int^{t_{k+1}}_{t_k}\int^s_{t_k}f'_N\left(u^N(t_k)\right)\int^{r}_{t_k}\cos\left(\lambda_i^{\frac{\alpha}{2}}(r-y)\right)\mathrm{d}\beta^i(y)\mathrm{d}r\mathrm{d}s\right]=0,\ j\ne k.
\end{eqnarray*}
Then
\begin{eqnarray*}
I_1&\lesssim&\tau^{\frac{\gamma}{\alpha}}\left(\frac{T^H}{\varepsilon}+\left\|u_0\right\|_{L^2(D,\dot{U}^\gamma)}+\left\|v_0\right\|_{L^2(D,\dot{U}^{\gamma-\alpha})}+\left\|u_0\right\|^2_{L^4(D,\dot{U}^\gamma)}+\left\|v_0\right\|^2_{L^4(D,\dot{U}^{\gamma-\alpha})}\right).
\end{eqnarray*}
Similar to $J_2$ and $I_1$, we have
\begin{eqnarray*}
I_2&\lesssim&\tau^{\frac{\gamma}{\alpha}}\left(\frac{T^H}{\varepsilon}+\left\|u_0\right\|_{L^2(D,\dot{U}^\gamma)}+\left\|v_0\right\|_{L^2(D,\dot{U}^{\gamma-\alpha})}+\left\|u_0\right\|^2_{L^4(D,\dot{U}^\gamma)}\right.\\
&&+\left.\left\|v_0\right\|^2_{L^4(D,\dot{U}^{\gamma-\alpha})}\right)+\tau\sum^{m}_{j=0}\left\|u^N(t_j)-u^N_j\right\|_{L^2(D,U)}.
\end{eqnarray*}
If $\gamma>2\alpha$, then
\begin{eqnarray*}
I_1&\lesssim&\tau^{2}\left(T^H+\left\|u_0\right\|_{L^2(D,\dot{U}^\gamma)}+\left\|v_0\right\|_{L^2(D,\dot{U}^{\gamma-\alpha})}+\left\|u_0\right\|^2_{L^4(D,\dot{U}^\gamma)}+\left\|v_0\right\|^2_{L^4(D,\dot{U}^{\gamma-\alpha})}\right)
\end{eqnarray*}
and
\begin{eqnarray*}
I_2&\lesssim&\tau^{2}\left(T^H+\left\|u_0\right\|_{L^2(D,\dot{U}^\gamma)}+\left\|v_0\right\|_{L^2(D,\dot{U}^{\gamma-\alpha})}+\left\|u_0\right\|^2_{L^4(D,\dot{U}^\gamma)}+\left\|v_0\right\|^2_{L^4(D,\dot{U}^{\gamma-\alpha})}\right)\\
&&+\tau\sum^{m}_{j=0}\left\|u^N(t_j)-u^N_j\right\|_{L^2(D,U)}.
\end{eqnarray*}
Using the above estimates and the discrete Gr\"onwall inequality, we obtain 
\begin{eqnarray}\label{eq:C3}
&&\left\|u^N(t_{m+1})-u^N_{m+1}\right\|_{L^2(D,U)}\\
&&\lesssim\tau^{\frac{\gamma}{\alpha}}\left(\frac{T^H}{\varepsilon}+\left\|u_0\right\|_{L^2(D,\dot{U}^\gamma)}+\left\|v_0\right\|_{L^2(D,\dot{U}^{\gamma-\alpha})}\right.\nonumber\\
&&~~~+\left.\left\|u_0\right\|^2_{L^4(D,\dot{U}^\gamma)}+\left\|v_0\right\|^2_{L^4(D,\dot{U}^{\gamma-\alpha})}\right), \alpha<\gamma\le2\alpha\nonumber
\end{eqnarray}
and
\begin{eqnarray}\label{eq:C4}
&&\left\|u^N(t_{m+1})-u^N_{m+1}\right\|_{L^2(D,U)}\\
&&\lesssim\tau^{2}\left(T^H+\left\|u_0\right\|_{L^2(D,\dot{U}^\gamma)}+\left\|v_0\right\|_{L^2(D,\dot{U}^{\gamma-\alpha})}\right.\nonumber\\
&&~~~+\left.\left\|u_0\right\|^2_{L^4(D,\dot{U}^\gamma)}+\left\|v_0\right\|^2_{L^4(D,\dot{U}^{\gamma-\alpha})}\right),\gamma>2\alpha.\nonumber
\end{eqnarray}
Take $0<\tau<1$ and $\varepsilon=\frac{1}{|\log(\tau)|}$. Combining above estimates and Theorem \ref{th:5}, we obtain the desired results.
\end{proof}

\section{Definitions of the cosine and sine operators}\label{sec:A}

In term of the eigenpairs $\left\{\left(\lambda_i,\phi_i\right)\right\}_{i=1}^\infty$, the cosine and sine operators can be expressed as
\begin{eqnarray*}
\sin\left(A^\alpha t\right)u(t)
&=&\sum^\infty_{i=1}\sin\left(\lambda_i^\alpha t\right)\left\langle u(t),\phi_{i}(x)\right\rangle\phi_{i}(x)\\
&=&\sum^\infty_{i=1}\sum^\infty_{j=1}(-1)^{j-1}\frac{\left(\lambda_i^\alpha t\right)^{2j-1}}{(2j-1)!}\left\langle u(t),\phi_{i}(x)\right\rangle\phi_{i}(x)
\end{eqnarray*}
and
\begin{eqnarray*}
\cos\left(A^\alpha t\right)u(t)
&=&\sum^\infty_{i=1}\cos\left(\lambda_i^\alpha t\right)\left\langle u(t),\phi_{i}(x)\right\rangle\phi_{i}(x)\\
&=&\sum^\infty_{i=1}\sum^\infty_{j=0}(-1)^{j}\frac{\left(\lambda_i^\alpha t\right)^{2j}}{(2j)!}\left\langle u(t),\phi_{i}(x)\right\rangle\phi_{i}(x).
\end{eqnarray*}

\section{Simulation of stochastic integral for fBm}\label{sec:B}

Suppose $0\le t_1\le \dots\le t_{m}\le \dots\le t_{M}=T$ $(m=1,2,\dots, M-1)$ and the fixed sizes of the mesh $\tau=t_{m+1}-t_{m}$. Let's consider the following vector
\begin{equation*}
Z=\left(\int^{t_{1}}_{0}s\mathrm{d}\beta_H(s),\int^{t_{2}}_{t_{1}}(s-t_1)\mathrm{d}\beta_H(s),\dots,
\int^{t_{M}}_{t_{M-1}}(s-t_{M-1})\mathrm{d}\beta_H(s)\right).
\end{equation*}
The stochastic integral $\int^{t_{m+1}}_{t_{m}}(s-t_m)\mathrm{d}\beta_H(s)$ is a Gaussian process with mean 0. The Cholesky method can be applied to stationary and non-stationary Gaussian processes. Thus, we use
the Ckolesky method to simulate \eqref{eq:5.9-0}.
The probability distribution of the vector $Z$ is normal with mean 0 and the covariance matrix $\Sigma$. Let $\Sigma_{i,j}$ be the element of row $i$, column $j$ of matrix $\Sigma$. Then
\begin{eqnarray*}
\Sigma_{i,j}&=&\mathrm{E}\left[\int^{t_{j+1}}_{t_j}(s-t_j)\mathrm{d}\beta_H(s)\int^{t_{k+1}}_{t_{k}}(t-t_k)\mathrm{d}\beta_H(t)\right].
\end{eqnarray*}
By using Lemma \ref{le:02}, for $j>k$, we have
\begin{eqnarray*}
&\mathrm{E}&\left[\int^{t_{j+1}}_{t_j}(s-t_j)\mathrm{d}B_H(s)\int^{t_{k+1}}_{t_{k}}(t-t_k)\mathrm{d}B_H(t)\right]\nonumber\\
&=&H(2H-2)\int^{t_{j+1}}_{t_j}\int^{t_{k+1}}_{t_{k}}(s-t_j)(t-t_k)(s-t)^{2H-2}\mathrm{d}t\mathrm{d}s\nonumber\\
&=&-\frac{\tau^2}{2}(t_{j+1}-t_{k+1})^{2H}+\frac{\tau}{2(2H+1)}\left((t_{j+1}-t_{k})^{2H+1}-(t_{j}-t_{k+1})^{2H+1}\right)\nonumber\\
&&-\frac{1}{2(2H+1)(2H+2)}\left((t_{j+1}-t_{k})^{2H+2}-2(t_{j}-t_{k})^{2H+2}+(t_{j}-t_{k+1})^{2H+2}\right)\\
&=&-\frac{\tau^{2+2H}}{2}(j-k)^{2H}+\frac{\tau^{2+2H}}{2(2H+1)}\left((j+1-k)^{2H+1}-(j-k-1)^{2H+1}\right)\nonumber\\
&&-\frac{\tau^{2+2H}}{2(2H+1)(2H+2)}\left((j+1-k)^{2H+2}-2(j-k)^{2H+2}+(j-k-1)^{2H+2}\right).
\end{eqnarray*}
When $j=k$,
\begin{eqnarray*}
\mathrm{E}\left[\int^{t_{j+1}}_{t_j}(s-t_j)\mathrm{d}B_H(s)\int^{t_{j+1}}_{t_{j}}(t-t_j)\mathrm{d}B_H(t)\right]=\frac{\tau^{2H+2}}{2H+2}.
\end{eqnarray*}
When $\Sigma$ is a symmetric positive matrix, the covariance matrix $\Sigma$ can be written as $L(M)L(M)'$, where the matrix $L(M)$ is lower triangular matrix and the matrix $L(M)'$ is the transpose of $L(M)$. Let $V=(V_1,V_2,\dots,V_M)$. The elements of the vector $V$ are a sequence of independent and identically distributed standard normal random variables. Since $Z=L(M)V$, then $Z$  can be simulated.
Let $l_{i,j}$ be the element of row $i$, column $j$ of matrix $L(M)$. That is,
\begin{equation*}
\Sigma_{i,j}=\sum^j_{k=1}l_{i,k}l_{j,k}, \quad j\le i.
\end{equation*}
As $i=j=1$, we have $l^2_{1,1}=\Sigma_{1,1}$. The $l_{i,j}$ satisfies
\begin{eqnarray*}
l_{i+1,1}&=&\frac{\Sigma_{i+1,1}}{l_{1,1}},\\ l^2_{i+1,i+1}&=&\Sigma_{i+1,i+1}-\sum^{i}_{k=1}l^2_{i+1,k},\\
l_{i+1,j}&=&\frac{1}{l_{j,j}}\left(\Sigma_{i+1,j}-\sum^{j-1}_{k=1}l_{i+1,k}l_{j,k}\right),\quad1<j\le i.
\end{eqnarray*}

\bibliographystyle{siamplain}

\end{document}